\documentclass[11pt, english]{article}
\usepackage[T1]{fontenc}
\usepackage[utf8]{inputenc}
\usepackage{textcomp}
\usepackage{amsmath}
\usepackage{amssymb}
\usepackage{cancel}
\usepackage{babel}
\usepackage{algorithm}
\usepackage{algpseudocode}

\usepackage{amsmath,amssymb,amsthm,mathrsfs,hyperref,color}

\DeclareMathAlphabet{\mathpzc}{OT1}{pzc}{m}{it}
\usepackage[nameinlink]{cleveref}
\hypersetup{colorlinks={true},linkcolor={blue},citecolor=blue}

\hypersetup{
    colorlinks=true,
    linkcolor=blue!60!black,
    citecolor=blue!60!black,
    urlcolor=blue!60!black
}

\usepackage{mathtools}
\usepackage[x11names]{xcolor}
\usepackage{tikz}
\usetikzlibrary{arrows.meta,calc,positioning}
\newtagform{blue}{\color{blue}(}{)}

\usepackage{aliascnt}

\theoremstyle{plain}
\newtheorem{theorem}{Theorem}[section]

\newaliascnt{proposition}{theorem}
\newtheorem{proposition}[proposition]{Proposition}
\aliascntresetthe{proposition}

\newaliascnt{lemma}{theorem}
\newtheorem{lemma}[lemma]{Lemma}
\aliascntresetthe{lemma}

\newaliascnt{corollary}{theorem}
\newtheorem{corollary}[corollary]{Corollary}
\aliascntresetthe{corollary}

\newaliascnt{assumption}{theorem}
\newtheorem{assumption}[assumption]{Assumption}
\aliascntresetthe{assumption}

\theoremstyle{definition}
\newaliascnt{definition}{theorem}
\newtheorem{definition}[definition]{Definition}
\aliascntresetthe{definition}

\newaliascnt{remark}{theorem}
\newtheorem{remark}[remark]{Remark}
\aliascntresetthe{remark}

\newaliascnt{example}{theorem}
\newtheorem{example}[example]{Example}
\aliascntresetthe{example}

\newcommand{\dist}{\operatorname{dist}}

\numberwithin{equation}{section}
\usepackage{authblk}

 \title{Semi-discrete approximation of Ma\~n\'e, Aubry and Mather sets}

\author[1]{Fabio Camilli}
\author[2]{Cristian Mendico}
\affil[1]{Dip. di Ingegneria e Geologia, Univ. "G. D'Annunzio" Chieti-Pescara \\ \href{mailto:fabio.camilli@unich.it}{fabio.camilli@unich.it}}
\affil[2]{Institut de Math\'ematique de Bourgogne - UMR 5584 CNRS, Universit\'e Bourgogne Europe\\ \href{mailto:cristian.mendico@u-bourgogne.fr}{cristian.mendico@u-bourgogne.fr}}

\begin{document}

	\maketitle

	\begin{abstract}
  We study a variational discretization in time of Aubry--Mather theory for
  Tonelli Lagrangians on the flat torus. We define discrete Ma\~n\'e, Aubry
  and Mather sets, establish their compactness and inclusion properties,
  and investigate their convergence as the time step tends to zero.
  We prove upper-limit results for the Ma\~n\'e sets and, when the continuous
  Aubry set has finitely many static classes, for the Aubry sets. We obtain
  full Kuratowski and Hausdorff convergence of the Aubry sets when this
  set is a single equilibrium or periodic orbit, and, under additional
  assumptions, a finite union of equilibria. Every weak limit of discrete
  Mather measures is a continuous Mather measure. For equilibrium Aubry
  sets, we give sufficient conditions for exact preservation of
  Mather sets and measure families at small time steps. A mechanical
  example illustrates these results and exhibits exact preservation
  of the Aubry set as well.

  \vspace{0.25cm}
  \\
  \noindent\textbf{2020 AMS:} 35F21; 37J06; 37J51; 49L25; 49M25.
\end{abstract}


\section{Introduction}
Aubry--Mather theory and weak KAM theory describe the global minimizing
dynamics of Tonelli Lagrangian systems and the structure of solutions of the
critical Hamilton--Jacobi equation
\[
  H(x,Du)=\alpha(H),\qquad x\in\mathbb T^d,
\]
where $\alpha(H)$ is Ma\~n\'e's critical value
\cite{Mather1991,Fathi2008,FathiSiconolfi,Sorrentino2015}; see also the
pioneering works \cite{Aubry1983,Bangert1988} and
\cite{Contreras2001,Mather2003}. The central objects of the
theory are three compact invariant subsets of the phase space, related by
$\widetilde{\mathcal M}_L\subset\widetilde{\mathcal A}_L\subset
\widetilde{\mathcal N}_L$. The Ma\~n\'e set $\widetilde{\mathcal N}_L$
consists of semistatic orbits, which minimize the critical action between any
two of their points. The Aubry set $\widetilde{\mathcal A}_L$ consists of
static orbits, which are also minimizing in the reverse direction. The Mather
set $\widetilde{\mathcal M}_L$ is the closure of the supports of Mather
measures, the invariant probability measures minimizing the Lagrangian
action.

These sets are defined through global minimization and are very sensitive to
perturbations. Approximating them is therefore a different problem from
approximating the critical value or a viscosity solution of the
Hamilton--Jacobi equation. Solutions are in general only Lipschitz, and
standard schemes for first order equations, such as semi-Lagrangian or
probabilistic dynamic programming methods, may smear out singularities or
introduce artificial selection mechanisms in the critical regime
\cite{FalconeFerretti,Camilli,Gomes_Ob}. A
natural alternative is a discretization which preserves the variational
structure. In the discrete Aubry--Mather model
\cite{Gomes,Garibaldi,Su,Zavidovique}, continuous curves are replaced by
sequences of points and the action by the sum
$\sum_k\tau L\bigl(x_k,(\widetilde x_{k+1}-\widetilde x_k)/\tau\bigr)$,
where $\widetilde x_k$ are chosen lifts of $x_k$ to $\mathbb R^d$ and
$\tau>0$ is the time step. The critical value is replaced by the additive eigenvalue
$\bar L(\tau)$ of the discrete Lax--Oleinik equation
\[
  u_\tau(y)+\tau\bar L(\tau)
  =\min_{x\in\mathbb T^d}\bigl\{u_\tau(x)+\widehat{\mathcal L}_\tau(x,y)\bigr\},
\]
where $\widehat{\mathcal L}_\tau(x,y)$ is the minimal one-step action from
$x$ to $y$. Su and Thieullen \cite{Su} proved that
$\bar L(\tau)\to-\alpha(H)$ and studied the convergence of discrete weak KAM
solutions as $\tau\to0$; see also \cite{Iturriaga-Wang} for applications to
mean field games. The behavior of the minimizing sets themselves under this
discretization is the subject of the present paper.

For fixed $\tau$ we introduce discrete Ma\~n\'e, Aubry and Mather sets
$\widetilde{\mathcal N}_L^\tau$, $\widetilde{\mathcal A}_L^\tau$ and
$\widetilde{\mathcal M}_L^\tau$ in the phase space
$\mathbb T^d\times\mathbb R^d$. They are defined through bi-infinite
semistatic and static discrete configurations and through minimizing
$\tau$-holonomic measures. On the torus, a discrete transition is determined
by its endpoints together with a winding number, which records the lifted
displacement. We show that these sets are nonempty and compact, contained in
a common compact set, and nested as in the continuous theory
(Prop. \ref{prop:discrete_set_inclusions}). We then study their convergence as
$\tau\to0$, in the Kuratowski and Hausdorff senses. Our main results can be
summarized as follows.

\emph{Ma\~n\'e set.} Discrete semistatic configurations satisfy a discrete
Euler--Lagrange equation, which gives velocity increments of order $\tau$.
Their interpolations are therefore compact, and their limits are semistatic
Euler--Lagrange trajectories. As a consequence,
$\limsup_{\tau\to0}\widetilde{\mathcal N}_L^\tau\subset\widetilde{\mathcal N}_L$,
with no further assumption (\Cref{thm:discrete_mane_limsup}).

\emph{Aubry set.} Compactness of static configurations gives semistatic
limits, but does not by itself preserve the reverse minimality. We
obtain $\limsup_{\tau\to0}\widetilde{\mathcal A}_L^\tau\subset
\widetilde{\mathcal A}_L$ when the continuous Aubry set has finitely many
static classes (\Cref{thm:limsup_aubry}). The proof adapts Bernard's level-separation
argument \cite{BernardConley2010}, using a strict critical subsolution
and discrete circuits with small calibration defect. The lower limit is more
delicate, because discrete static points might concentrate near a proper
subfamily of components. We prove full Kuratowski and Hausdorff convergence
when $\widetilde{\mathcal A}_L$ is a single equilibrium or periodic orbit
(\Cref{thm:single_component}); the discrete Ma\~n\'e sets also converge in
this case. For finitely many equilibrium components, we obtain Aubry
convergence under a quadratic critical gap on a compact neighborhood of
the continuous Aubry set (\Cref{thm:convergence_equilibria}). Under the
gap-region hypotheses of that theorem, the discrete critical
value is exact, $\bar L(\tau)=-\alpha(H)$ for $\tau$ small, and
$\widetilde{\mathcal A}_L\subset\widetilde{\mathcal A}_L^\tau$. The case of
several periodic components remains open, and we explain why the
equilibrium mechanism does not extend to it (\Cref{rem:periodic_open}).

\emph{Mather measures and Mather set.} Every weak limit of exact discrete
Mather measures is a continuous Mather measure
(\Cref{prop:mather_measure_limsup}). In general, the corresponding lower-limit
statement is open, since the discretization may select a proper subset of the
continuous Mather measures. For the lifted Mather sets we obtain the upper
limits inherited from the Ma\~n\'e and Aubry results. When the Aubry set
consists of equilibria, the quadratic gap on the region specified in
\Cref{thm:convergence_equilibria} implies
exact equality of the discrete and continuous Mather sets and measure
families for small $\tau$
(\Cref{cor:discrete_mather_upper_limits,cor:mather_equilibrium_convergence}).
If the continuous Mather measure is unique and its support equals the
Ma\~n\'e set, all three discrete sets converge fully
(\Cref{cor:full_convergence_unique_mather}).

A complementary fully discrete approximation of Mather measures and
sets is studied in \cite{CamilliMendicoFullyDiscrete}, where both time
and space are discretized and a finite-dimensional procedure for
approximating the Mather set is developed. Here we retain continuous
spatial variables and study the semi-discrete Ma\~n\'e, Aubry and
Mather sets.

\smallskip

The paper is organized as follows. \Cref{sec:assumptions} recalls the
continuous theory. \Cref{sec:discrete} introduces the discrete framework and
the discrete Ma\~n\'e, Aubry and Mather sets. \Cref{sec:conv_mane},
\Cref{sec:conv_aubry} and \Cref{sec:conv_mather} are devoted to the
convergence of the Ma\~n\'e set, of the Aubry set and of the Mather set and
Mather measures, respectively. The mechanical examples and their applications to the preceding convergence results are collected in \Cref{sec:mechanical_model}. \Cref{sec:set_convergence} collects the notions of convergence of sets used throughout the paper.

\section{Continuous Aubry--Mather theory}\label{sec:assumptions}

We briefly recall the continuous Aubry--Mather framework for autonomous
Tonelli Lagrangians on the flat torus; see
\cite{Fathi2008,Sorrentino2015} for further details. Let
\[
H:\mathbb T^d\times\mathbb R^d\longrightarrow\mathbb R
\]
be a Tonelli Hamiltonian, that is, \(H\in C^2\), \(D^2_{pp}H(x,p)\) is
positive definite, and \(H\) is uniformly superlinear in \(p\). Its
Legendre transform is
\[
L(x,v):=\sup_{p\in\mathbb R^d}\{p\cdot v-H(x,p)\}.
\]
Then \(L\in C^2(\mathbb T^d\times\mathbb R^d)\), is strictly convex in
\(v\), and is superlinear. In particular, for every \(K>0\) there exists
\(C_K>0\) such that
\begin{equation}\label{superlinearity}
	L(x,v)\geq K|v|-C_K
	\qquad
	\forall (x,v)\in\mathbb T^d\times\mathbb R^d.
\end{equation}

The Euler--Lagrange flow associated with \(L\) is denoted by
\[
\phi_t^L(x,v)
=
\bigl(\gamma_{x,v}(t),\dot\gamma_{x,v}(t)\bigr),
\]
where \(\gamma_{x,v}\) solves
\[
\frac{d}{dt}D_vL(\gamma_{x,v},\dot\gamma_{x,v})
=
D_xL(\gamma_{x,v},\dot\gamma_{x,v}).
\]
Since the base manifold is compact and \(L\) is Tonelli, this flow is
complete.

Let \(\alpha(H)\) be the Ma\~n\'e critical value, namely the unique constant
for which the Hamilton--Jacobi equation
\begin{equation}
	H(x,Du)=\alpha(H)
	\qquad\text{in }\mathbb T^d
\end{equation}
admits a periodic viscosity solution. For \(t>0\), define the minimal
action
\begin{equation}
	h_t(x,y):=
	\inf_{\substack{\gamma\in W^{1,1}([0,t];\mathbb T^d)\\
			\gamma(0)=x,\ \gamma(t)=y}}
	\int_0^t L(\gamma(s),\dot\gamma(s))\,ds
\end{equation}
and the critical action potential
\begin{equation}
	\Phi_L(x,y):=
	\inf_{t>0}\bigl\{h_t(x,y)+\alpha(H)t\bigr\}.
\end{equation}

\begin{definition}
	An Euler--Lagrange trajectory
	\(\gamma:\mathbb R\to\mathbb T^d\) is called \emph{semistatic} if
	\begin{equation}\label{eq:semistatic}
		\int_a^b
		\bigl(L(\gamma(s),\dot\gamma(s))+\alpha(H)\bigr)\,ds
		=
		\Phi_L(\gamma(a),\gamma(b))
	\end{equation}
	for every \(a<b\). It is called \emph{static} if it is semistatic and, in
	addition,
	\begin{equation}\label{eq:static}
		\Phi_L(\gamma(a),\gamma(b))
		+
		\Phi_L(\gamma(b),\gamma(a))
		=
		0
	\end{equation}
	for every \(a<b\).
\end{definition}

Thus, semistatic trajectories minimize the critical action between any two
of their points, while static trajectories also satisfy the corresponding
minimality condition in the reverse direction.

\begin{definition}
	The lifted Ma\~n\'e and Aubry sets are defined by
	\begin{align}
		\widetilde{\mathcal N}_L
		&:=
		\left\{
		(x,v)\in\mathbb T^d\times\mathbb R^d:
		\gamma_{x,v}\text{ is semistatic on }\mathbb R
		\right\},
		\\
		\widetilde{\mathcal A}_L
		&:=
		\left\{
		(x,v)\in\mathbb T^d\times\mathbb R^d:
		\gamma_{x,v}\text{ is static on }\mathbb R
		\right\}.
	\end{align}
	Their projections onto \(\mathbb T^d\) are denoted by
	\(\mathcal N_L\) and \(\mathcal A_L\), respectively.
\end{definition}

For comparison with the discrete setting, recall that a curve \(\gamma\) is
\((u,L)\)-calibrated on an interval \(I\), for a critical weak KAM solution
\(u\), if
\begin{equation}
	u(\gamma(b))-u(\gamma(a))
	=
	\int_a^b
	\bigl(L(\gamma(s),\dot\gamma(s))+\alpha(H)\bigr)\,ds
\end{equation}
for every \(a<b\) in \(I\). Every semistatic trajectory is calibrated on
\(\mathbb R\) by a suitable critical weak KAM solution, whereas static
trajectories belong to the Aubry set independently of the chosen critical
solution.

Let \(\mathcal P_1(\mathbb T^d\times\mathbb R^d)\) denote the set of Borel
probability measures with finite first moment in the velocity variable.

\begin{definition}
	A measure
	\(\mu\in\mathcal P_1(\mathbb T^d\times\mathbb R^d)\) is called
	\emph{closed} if
	\begin{equation}
		\int_{\mathbb T^d\times\mathbb R^d}
		v\cdot D\varphi(x)\,d\mu(x,v)
		=
		0
		\qquad
		\forall\varphi\in C^1(\mathbb T^d).
	\end{equation}
	A closed measure is called a \emph{Mather measure} if it minimizes
	\(\int L\,d\mu\) among all closed probability measures.
\end{definition}

The minimum action among closed probability measures is \(-\alpha(H)\).
Consequently, \(\mu\) is a Mather measure if and only if it is closed and
\[
\int_{\mathbb T^d\times\mathbb R^d}L(x,v)\,d\mu(x,v)
=
-\alpha(H).
\]
Every minimizing closed measure is invariant under the Euler--Lagrange
flow, so this definition agrees with the classical one in terms of
invariant minimizing measures \cite{Fathi2008,Sorrentino2015}.

\begin{definition}
	The lifted Mather set is
	\[
	\widetilde{\mathcal M}_L
	:=
	\overline{
		\bigcup
		\left\{
		\operatorname{spt}(\mu):
		\mu\text{ is a Mather measure}
		\right\}
	}
	\subset\mathbb T^d\times\mathbb R^d,
	\]
	and its projection onto \(\mathbb T^d\) is denoted by
	\(\mathcal M_L\).
\end{definition}

The critical potential $\Phi_L$ is finite and Lipschitz, satisfies the
triangle inequality, and vanishes on the diagonal. In particular,
$\Phi_L(x,y)+\Phi_L(y,x)\geq0$. We shall also use the following standard
properties.

\begin{proposition}\label{prop:continuous_sets}
	The sets
	\(\widetilde{\mathcal M}_L\),
	\(\widetilde{\mathcal A}_L\), and
	\(\widetilde{\mathcal N}_L\) are nonempty, compact, and invariant under
	\(\phi_t^L\). Moreover,
	\begin{equation}\label{eq:continuous_inclusions}
		\widetilde{\mathcal M}_L
		\subset
		\widetilde{\mathcal A}_L
		\subset
		\widetilde{\mathcal N}_L.
	\end{equation}
	The projection $\pi:\widetilde{\mathcal A}_L\longrightarrow\mathcal A_L$
	is injective, and \(\widetilde{\mathcal A}_L\) is a Lipschitz graph over
	\(\mathcal A_L\).
	In particular, every orbit contained in the Mather or Aubry set is globally
	minimizing. 
\end{proposition}

\begin{lemma}\label{lem:single_static_orbit}
  If $\widetilde{\mathcal A}_L$ is a single equilibrium or a single
  nonstationary periodic orbit, then
  $\widetilde{\mathcal N}_L=\widetilde{\mathcal A}_L$.
\end{lemma}

\begin{proof}
  Let $\gamma$ be a semistatic trajectory. Its lifted orbit is contained in
  the compact set $\widetilde{\mathcal N}_L$. The standard limit-set
  property for semistatic trajectories states that its $\alpha$- and
  $\omega$-limit sets are contained in $\widetilde{\mathcal A}_L$; see
  \cite[Lemma~3.1]{ContrerasPaternain2002}. Thus there exist
  $s_n\to-\infty$, $t_n\to+\infty$ and $x_-,x_+\in\mathcal A_L$ with
  $\gamma(s_n)\to x_-$ and $\gamma(t_n)\to x_+$. Since both limit points
  lie on the same static orbit,
  $\Phi_L(x_-,x_+)+\Phi_L(x_+,x_-)=0$.

  Fix $a<b$ and write
  $A_{a,b}:=\int_a^b(L(\gamma,\dot\gamma)+\alpha(H))\,dt$.
  For $s_n<a<b<t_n$, the triangle inequality for $\Phi_L$ and
  semistaticity give
  \begin{align*}
    \Phi_L(\gamma(b),\gamma(a))
    &\leq \Phi_L(\gamma(b),\gamma(t_n))
      +\Phi_L(\gamma(t_n),\gamma(s_n))
      +\Phi_L(\gamma(s_n),\gamma(a))\\
    &=\Phi_L(\gamma(s_n),\gamma(t_n))
      +\Phi_L(\gamma(t_n),\gamma(s_n))-A_{a,b}.
  \end{align*}
  By continuity of $\Phi_L$, letting $n\to\infty$ yields
  $\Phi_L(\gamma(b),\gamma(a))\leq-A_{a,b}$.
  The reverse inequality follows from
  $\Phi_L(\gamma(a),\gamma(b))+\Phi_L(\gamma(b),\gamma(a))\geq0$
  and $A_{a,b}=\Phi_L(\gamma(a),\gamma(b))$. Hence $\gamma$ is static.
  The opposite set inclusion is \eqref{eq:continuous_inclusions}.
\end{proof}


%
%
%


\section{Discrete Aubry--Mather theory}\label{sec:discrete}

We introduce the semi-discrete Aubry--Mather framework associated
with the time discretization considered in \cite{Garibaldi,Su}.
 The notions of Kuratowski and Hausdorff
convergence used below are recalled in \Cref{sec:set_convergence}. 

Throughout this section, $\tau\in(0,1)$ is fixed,
$\pi:\mathbb R^d\to\mathbb T^d$ denotes the canonical projection, and
\[
T_\tau(x,v):=x+\tau v
\qquad\text{in }\mathbb T^d.
\]
We equip \(\mathbb T^d\) with the quotient distance
\[
d_{\mathbb T^d}(x,y)
:=
\min\bigl\{|w|:\ w\in\mathbb R^d,\ \pi(w)=y-x\bigr\}.
\]
The minimum exists because \(\pi^{-1}(y-x)\) is a translate of
\(\mathbb Z^d\), and
\[
d_{\mathbb T^d}(\pi(a),\pi(b))\le |a-b|
\qquad\text{for all }a,b\in\mathbb R^d.
\]
On \(X:=\mathbb T^d\times\mathbb R^d\) we use the product distance
\[
d_X\bigl((x,v),(y,w)\bigr)
:=
\bigl(d_{\mathbb T^d}(x,y)^2+|v-w|^2\bigr)^{1/2},
\]
and write \(X_R:=\mathbb T^d\times\overline B_R\), which is compact for
every \(R>0\).
Let $Q=[0,1)^d$ and let $\iota:\mathbb T^d\to Q$ be the Borel section
associated with the standard fundamental domain. For
$x,y\in\mathbb T^d$ and $\ell\in\mathbb Z^d$, define
\begin{align}\label{eq:velocity_action_winding}
	\Delta_\ell(x,y)
	&:=\iota(y)-\iota(x)+\ell,
	\qquad v_\tau(x,y,\ell)
	:=\frac{\Delta_\ell(x,y)}{\tau},
	\\
	\mathcal L_\tau(x,y,\ell)
	&:=\tau L\bigl(x,v_\tau(x,y,\ell)\bigr).
	\nonumber
\end{align}
The winding number $\ell$ distinguishes transitions having the same
projected endpoints but different lifted displacements. Indeed, the
endpoints of a transition on $\mathbb T^d$ do not determine its
velocity, and the corresponding displacement in the universal cover
must therefore be retained.

The correspondence
\[
(x,y,\ell)
\longmapsto
\bigl(x,v_\tau(x,y,\ell)\bigr)
\]
is a Borel bijection from
$\mathbb T^d\times\mathbb T^d\times\mathbb Z^d$ onto
$\mathbb T^d\times\mathbb R^d$. Its inverse associates with $(x,v)$
the endpoint $y=T_\tau(x,v)$ and the unique winding number
$\ell=\tau v-\iota(y)+\iota(x)\in\mathbb Z^d$.

Since $\iota$ is not continuous, topological arguments will be carried
out either in phase-space variables or through the reduced action
\begin{equation}\label{eq:reduced_cost}
	\widehat{\mathcal L}_\tau(x,y)
	:=\min_{\ell\in\mathbb Z^d}
		\mathcal L_\tau(x,y,\ell)=\tau\min_{\substack{w\in\mathbb R^d\\
	                     \pi(w)=y-x}}
		L\left(x,\frac{w}{\tau}\right).
\end{equation}

Set
\begin{equation}\label{eq:standing_constants}
	A:=\max_{\mathbb T^d\times\overline B_1}|L|,
	\qquad
	C_0:=\max\left\{
		0,-\min_{\mathbb T^d\times\mathbb R^d}L
	\right\},
	\qquad
	K:=A+C_0+1.
\end{equation}
Applying the superlinearity estimate \eqref{superlinearity} with
this value of $K$, we fix a constant $C_K>0$ such that
\[
L(x,v)\geq K|v|-C_K
\qquad
\forall (x,v)\in\mathbb T^d\times\mathbb R^d.
\]

\begin{lemma}\label{lem:reduced_cost}
	For every $x,y\in\mathbb T^d$, the minima in
	\eqref{eq:reduced_cost} are attained. Moreover,
	$\widehat{\mathcal L}_\tau$ is continuous and independent of the
	choice of the Borel section $\iota$.
\end{lemma}

\begin{proof} As $\ell$ ranges over $\mathbb Z^d$, the displacements $\Delta_\ell(x,y)$ range over $\pi^{-1}(y-x)$. This proves the second identity in \eqref{eq:reduced_cost} and shows that $\widehat{\mathcal L}_\tau$ is independent of $\iota$. Since \[ \tau L\left(x,\frac{w}{\tau}\right)\longrightarrow+\infty \qquad\text{as }|w|\to\infty \] uniformly in $x$, the minimum is attained and all minimizing displacements lie in a fixed compact set. Finally, after choosing local continuous lifts of $x$ and $y$, only finitely many winding numbers need to be considered. Thus $\widehat{\mathcal L}_\tau$ is locally the minimum of finitely many continuous functions and is therefore continuous. \end{proof}

We consider the additive eigenvalue problem
\begin{equation}\label{discrete_laxoleinik_classical}
	u_\tau(y)+\tau\bar L(\tau)
	=
	\min_{x\in\mathbb T^d}
	\left\{
		u_\tau(x)+\widehat{\mathcal L}_\tau(x,y)
	\right\},
	\qquad y\in\mathbb T^d.
\end{equation}
This is equivalent to the formulation on the universal cover used in
\cite{Su}. The following result reformulates
\cite[Theorem~9]{Su} in the present notation and records elementary
bounds that will be used repeatedly. The rate \eqref{eq:critical_value_rate}
is also proved directly in \Cref{lem:rate} below.

\begin{theorem}
	There exists a unique constant $\bar L(\tau)\in\mathbb R$ for
	which \eqref{discrete_laxoleinik_classical} admits a continuous
	solution $u_\tau$. Moreover,
	\begin{equation}\label{conv_ergodic_const}
		\bar L(\tau)\longrightarrow-\alpha(H)
		\qquad\text{as }\tau\to0.
	\end{equation}
	More precisely, there exists a constant $C_L>0$, independent of
	$\tau$, such that
	\begin{equation}\label{eq:critical_value_rate}
		\bigl|\bar L(\tau)+\alpha(H)\bigr|
		\leq C_L\tau,
		\qquad \tau\in(0,1).
	\end{equation}
	In addition,
	\begin{equation}\label{eq:ergodic_constant_bounds}
		-C_0
		\leq
		\bar L(\tau)
		\leq
		\min_{x\in\mathbb T^d}L(x,0)
		\leq A.
	\end{equation}
	There exists $K_0>0$, independent of $\tau\in(0,1)$, such that
	$\operatorname{Lip}(u_\tau)\leq K_0$ for every solution. After the
	normalization $\min u_\tau=0$, every sequence with $\tau\to0$ has a
	uniformly convergent subsequence, and every limit is a continuous
	critical weak KAM solution.
\end{theorem}


We fix one solution $u_\tau$ of
\eqref{discrete_laxoleinik_classical}. The Ma\~n\'e, Aubry, and Mather
sets introduced below are independent of this choice. The notions of
calibrated transition and calibration defect, instead, refer to the
fixed solution $u_\tau$. Equation \eqref{discrete_laxoleinik_classical} implies
\begin{equation}\label{eq:LO_phase_space}
	u_\tau\bigl(T_\tau(x,v)\bigr)-u_\tau(x)
	\leq
	\tau L(x,v)-\tau\bar L(\tau)
\end{equation}
for every $(x,v)\in\mathbb T^d\times\mathbb R^d$.

A transition $(x,v)$ is said to be \emph{calibrated} if equality holds
in \eqref{eq:LO_phase_space}. Equivalently, it is calibrated if its
calibration defect
\begin{equation}
	G_\tau(x,v)
	:=
	\tau L(x,v)-\tau\bar L(\tau)
	-u_\tau\bigl(T_\tau(x,v)\bigr)+u_\tau(x)
\end{equation}
vanishes. By \eqref{eq:LO_phase_space}, the function $G_\tau$ is
continuous and nonnegative.

\subsection{Discrete Ma\~n\'e and Aubry sets}
A discrete configuration on an interval $I\subset\mathbb Z$ consists
of points $\{x_k\}_{k\in I}\subset\mathbb T^d$ and winding numbers
$\ell_k\in\mathbb Z^d$ whenever $k,k+1\in I$. Its velocities are
$v_k=v_\tau(x_k,x_{k+1},\ell_k)$. Define the critical costs
\begin{equation}
	c_\tau(x,y,\ell)
	:=
	\mathcal L_\tau(x,y,\ell)-\tau\bar L(\tau),
	\qquad
	\widehat c_\tau(x,y)
	:=
	\widehat{\mathcal L}_\tau(x,y)-\tau\bar L(\tau).
\end{equation} 
The discrete Ma\~n\'e potential is
\begin{equation}\label{eq:discrete_mane_potential}
	\Phi_\tau(x,y)
	:=
	\inf_{\substack{n\geq1,\ x_0=x,\ x_n=y\\
			x_1,\ldots,x_{n-1}\in\mathbb T^d\\
			\ell_0,\ldots,\ell_{n-1}\in\mathbb Z^d}}
	\sum_{k=0}^{n-1}c_\tau(x_k,x_{k+1},\ell_k)
	=
	\inf_{n\geq1}\Phi_\tau^{(n)}(x,y),
\end{equation}
where
\begin{equation}
	\Phi_\tau^{(n)}(x,y)
	:=
	\min_{x_1,\ldots,x_{n-1}\in\mathbb T^d}
	\sum_{k=0}^{n-1}\hat c_\tau(x_k,x_{k+1}),
	\qquad
	x_0=x,\quad x_n=y.
\end{equation}

\begin{definition}
	A bi-infinite configuration
	$\{(x_k,\ell_k)\}_{k\in\mathbb Z}$ is \emph{semistatic} if
	\begin{equation}\label{eq:discrete_semistatic}
		\sum_{k=m}^{n-1}c_\tau(x_k,x_{k+1},\ell_k)
		=
		\Phi_\tau(x_m,x_n)
		\qquad
		\forall m<n.
	\end{equation}
	It is \emph{static} if it is semistatic and
	\begin{equation}\label{eq:discrete_static}
		\Phi_\tau(x_m,x_n)+\Phi_\tau(x_n,x_m)=0
		\qquad
		\forall m<n.
	\end{equation}
	It is \emph{calibrated by $u_\tau$} if every transition
	$(x_k,v_k)$ is calibrated.
\end{definition}

Taking $n=m+1$ in \eqref{eq:discrete_semistatic} gives
\begin{equation}\label{eq:semistatic_optimal_winding}
	c_\tau(x_k,x_{k+1},\ell_k)
	=
	\hat c_\tau(x_k,x_{k+1})
	=
	\Phi_\tau(x_k,x_{k+1}),
\end{equation}
so every winding number occurring in a semistatic configuration is
optimal.

\begin{definition}\label{def:discrete_mane_aubry_sets}
	The lifted discrete Ma\~n\'e and Aubry sets are
	\begin{align*}
		\widetilde{\mathcal N}_L^\tau
		&:=
		\left\{
		(x,v):
		\begin{array}{l}
			\text{there exists a bi-infinite semistatic configuration}\\
			\{(x_k,\ell_k)\}_{k\in\mathbb Z}
			\text{ with }x_0=x,\
			v=v_\tau(x_0,x_1,\ell_0)
		\end{array}
		\right\},\\
		\widetilde{\mathcal A}_L^\tau
		&:=
		\left\{
		(x,v):
		\begin{array}{l}
			\text{there exists a bi-infinite static configuration}\\
			\{(x_k,\ell_k)\}_{k\in\mathbb Z}
			\text{ with }x_0=x,\
			v=v_\tau(x_0,x_1,\ell_0)
		\end{array}
		\right\}.
	\end{align*}
\end{definition}

We next study the compactness of the discrete Ma\~n\'e and Aubry sets. The main preliminary ingredient is a uniform a priori bound on the velocities of calibrated and semistatic configurations.

\begin{lemma}\label{lem:velocity_bound}
	Set
	\begin{equation}
		\kappa:=\max\{1,C_K+A\}.
	\end{equation}
	Then every calibrated transition and every transition belonging to a
	semistatic configuration satisfies 
	\[
	|v|\leq\kappa \quad \textrm{uniformly in}  \quad \tau\in(0,1).
	\]
\end{lemma}

\begin{proof}
	Let $(x,v)$ be one of these transitions and assume $|v|>1$. Set
	\[
	y=T_\tau(x,v), \quad \textrm{and} \quad n:=\lceil|v|\rceil, \quad \textrm{so that} \quad |v/n|\leq1, \;\; \textrm{and}\;\; n\leq|v|+1.
	\]
	
	Fix a lift $\hat x$ of $x$ and set
	$z_j:=\pi(\hat x+j\tau v/n)$ for $j=0,\ldots,n$. If $(x,v)$ is calibrated, summing \eqref{eq:LO_phase_space} along
	this chain gives
	\[
	\tau L(x,v)-\tau\bar L(\tau)
	\leq
	n\tau\bigl(A-\bar L(\tau)\bigr).
	\]
	The same inequality holds for a semistatic transition by
	\eqref{eq:semistatic_optimal_winding}, since the interpolating chain
	is admissible in the definition of $\Phi_\tau$. Hence
	\[
	L(x,v)
	\leq
	nA-(n-1)\bar L(\tau)
	\leq
	(A+C_0)|v|+A.
	\]
	Combining this with \eqref{superlinearity} yields
	$|v|\leq C_K+A\leq\kappa$.
\end{proof}

\begin{lemma}\label{lem:mane_potential}
	The function $\Phi_\tau$ is real valued and continuous. Moreover,
	\begin{equation}\label{eq:mane_domination}
		u_\tau(y)-u_\tau(x)
		\leq
		\Phi_\tau(x,y)
		\leq
		\hat c_\tau(x,y),
		\qquad
		\Phi_\tau(x,x)\geq0,
	\end{equation}
	and
	\begin{equation}\label{eq:mane_triangle}
		\Phi_\tau(x,z)
		\leq
		\Phi_\tau(x,y)+\Phi_\tau(y,z).
	\end{equation}
\end{lemma}

\begin{proof}
	The upper bound follows by taking one step, while the lower bound is
	obtained by summing the domination inequality along any finite chain.
	The triangle inequality follows by concatenation.
	
	Let $\omega_\tau$ be a common modulus of continuity of
	$\hat c_\tau$ in each variable. Changing only the endpoints of an
	optimal chain gives, for every $n\geq1$,
	\[
	\left|
	\Phi_\tau^{(n)}(x',y')-\Phi_\tau^{(n)}(x,y)
	\right|
	\leq
	\omega_\tau\bigl(d_{\mathbb T^d}(x,x')\bigr)
	+
	\omega_\tau\bigl(d_{\mathbb T^d}(y,y')\bigr).
	\]
	Taking the infimum over $n$ gives the corresponding one-sided
	estimate for $\Phi_\tau$; exchanging the two pairs of endpoints gives
	the reverse estimate. Thus $\Phi_\tau$ is continuous.
\end{proof}

\begin{lemma}\label{lem:closedness}
	The sets $\widetilde{\mathcal N}_L^\tau$ and
	$\widetilde{\mathcal A}_L^\tau$ are compact subsets of
	$\mathbb T^d\times\overline B_\kappa$, and
	$\widetilde{\mathcal A}_L^\tau
	\subset\widetilde{\mathcal N}_L^\tau$.
\end{lemma}

\begin{proof}
	The inclusion is immediate and the velocity bound follows from
	 Lemma \ref{lem:velocity_bound}. In phase-space variables, semistaticity is
	equivalent to
	\[
	\sum_{k=m}^{n-1}
	\bigl[\tau L(x_k,v_k)-\tau\bar L(\tau)\bigr]
	=
	\Phi_\tau(x_m,x_n)
	\qquad
	\forall m<n,
	\]
	together with $x_{k+1}=T_\tau(x_k,v_k)$. Staticity requires in
	addition \eqref{eq:discrete_static}. These are countable families of
	closed conditions.
	
	Given a convergent sequence of points in either set, choose
	corresponding configurations. The velocity bound and a diagonal
	extraction give pointwise convergence at every index. Compatibility
	and the semistatic or static identities pass to the limit.
	The limiting compatibility relation uniquely determines the winding
	numbers. Thus both sets are closed in the
	compact space $\mathbb T^d\times\overline B_\kappa$.
\end{proof}

\subsection{Discrete Mather measures}

\begin{definition}
	A probability measure
	$\mu\in\mathcal P(\mathbb T^d\times\mathbb R^d)$ is
	\emph{$\tau$-holonomic} if
	\begin{equation}\label{eq:discrete_holonomy}
		\int
		\bigl(\varphi(T_\tau(x,v))-\varphi(x)\bigr)\,d\mu(x,v)
		=
		0
		\qquad
		\forall\varphi\in C(\mathbb T^d).
	\end{equation}
	We denote by $\mathcal H_\tau$ the set of $\tau$-holonomic
	probability measures satisfying $\int L\,d\mu<+\infty$.
\end{definition}

\begin{definition}
	A measure $\mu_\tau\in\mathcal H_\tau$ is a
	\emph{discrete Mather measure} if it minimizes
	$\mathfrak A(\mu):=\int L\,d\mu$ over $\mathcal H_\tau$.
	The discrete Mather set is
	\[
	\widetilde{\mathcal M}_L^\tau
	:=
	\overline{
		\bigcup
		\left\{
		\operatorname{spt}(\mu_\tau):
		\mu_\tau\text{ is a discrete Mather measure}
		\right\}
	}.
	\]
\end{definition}

\begin{lemma}\label{lem:minimum_action}
	The minimum action is
	\begin{equation}
		\min_{\mu\in\mathcal H_\tau}\int L\,d\mu
		=
		\bar L(\tau),
	\end{equation}
	and it is attained. Moreover, every discrete Mather measure satisfies
	\begin{equation}\label{eq:zero_defect_support}
		G_\tau=0
		\quad\text{on }\operatorname{spt}(\mu_\tau),
		\qquad
		\operatorname{spt}(\mu_\tau)
		\subset
		\mathbb T^d\times\overline B_\kappa.
	\end{equation}
\end{lemma}

\begin{proof}
	Integrating \eqref{eq:LO_phase_space} against any
	$\mu\in\mathcal H_\tau$ gives
	$\int L\,d\mu\geq\bar L(\tau)$.
	
	For the converse, choose recursively a minimizing pair
	$(x_k,\ell_k)$ in \eqref{discrete_laxoleinik_classical} with terminal
	point $x_{k+1}$, obtaining a backward calibrated configuration. Set
	$v_k=v_\tau(x_k,x_{k+1},\ell_k)$ and
	\[
	\eta_N:=\frac1N\sum_{k=-N}^{-1}\delta_{(x_k,v_k)}.
	\]
	By \Cref{lem:velocity_bound}, the measures $\eta_N$ have a weakly
	convergent subsequence, with limit $\eta$. The telescoping identity
	\[
	\int
	\bigl(\varphi(T_\tau(x,v))-\varphi(x)\bigr)\,d\eta_N
	=
	\frac{\varphi(x_0)-\varphi(x_{-N})}{N}
	\]
	shows that $\eta$ is $\tau$-holonomic. Calibration gives
	\[
	\int L\,d\eta_N
	=
	\bar L(\tau)
	+
	\frac{u_\tau(x_0)-u_\tau(x_{-N})}{N\tau}.
	\]
	Since the measures are supported in a common compact set and
	$u_\tau$ is bounded, we obtain $\int L\,d\eta=\bar L(\tau)$.
	
	If $\mu_\tau$ is minimizing, then
	$\int G_\tau\,d\mu_\tau=0$. Since $G_\tau$ is continuous and
	nonnegative, it vanishes on the support. The velocity bound follows
	from \Cref{lem:velocity_bound}.
\end{proof}

The rate \eqref{eq:critical_value_rate} can be proved directly; we include the
argument, since it is used in \Cref{lem:static_small_defect_circuits}.

\begin{lemma}\label{lem:rate}
There is $C_L>0$ such that $|\bar L(\tau)+\alpha(H)|\le C_L\tau$ for every
$\tau\in(0,1)$.
\end{lemma}

\begin{proof}
Let $\Lambda$ be a Lipschitz constant of $L$ in $x$ on $X_R$, where
$R\ge\kappa$ is such that $\widetilde{\mathcal M}_L\subset X_R$.

($\ge$) Let $\eta$ be a discrete Mather measure, so that
$\operatorname{spt}\eta\subset X_\kappa$ by \Cref{lem:minimum_action}, and define
the probability measure $\nu$ by
\[
  \int f\,d\nu:=\int\frac1\tau\int_0^\tau f(x+sv,v)\,ds\,d\eta(x,v).
\]
For $\varphi\in C^1(\mathbb T^d)$,
$\int v\cdot D\varphi\,d\nu=\int\tau^{-1}\bigl(\varphi(T_\tau(x,v))-\varphi(x)\bigr)\,d\eta=0$,
so $\nu$ is closed and $\int L\,d\nu\ge-\alpha(H)$. Since
$\int L\,d\nu\le\int L\,d\eta+\Lambda\kappa\tau=\bar L(\tau)+\Lambda\kappa\tau$,
we get $\bar L(\tau)\ge-\alpha(H)-\Lambda\kappa\tau$.

($\le$) Let $\mu$ be a Mather measure; it is $\phi^L_t$-invariant and
supported in $X_R$. For $(x,v)\in X$ let $\widetilde\gamma_{x,v}$ be a lift of
$\gamma_{x,v}$ with $\widetilde\gamma_{x,v}(0)=\widetilde x$, and set
$\Psi_\tau(x,v):=\bigl(x,\tau^{-1}(\widetilde\gamma_{x,v}(\tau)-\widetilde x)\bigr)$,
$\mu_\tau:=(\Psi_\tau)_\#\mu$. Then
$\int\varphi\circ T_\tau\,d\mu_\tau=\int\varphi(\gamma_{x,v}(\tau))\,d\mu=\int\varphi\,d\mu
=\int\varphi\,d\mu_\tau$, so $\mu_\tau$ is $\tau$-holonomic. By Jensen's
inequality and $d_{\mathbb T^d}(\gamma_{x,v}(s),x)\le R\tau$ for $s\in[0,\tau]$,
\[
  \int L\,d\mu_\tau\le\int\frac1\tau\int_0^\tau L\bigl(\phi^L_s(x,v)\bigr)\,ds\,d\mu+\Lambda R\tau
  =-\alpha(H)+\Lambda R\tau,
\]
hence $\bar L(\tau)\le-\alpha(H)+\Lambda R\tau$. Take $C_L:=\Lambda R$.
\end{proof}

To prove that the discrete Mather set is contained in the discrete
Aubry set, we adapt the lifting and recurrence strategy of
Garibaldi and Thieullen
\cite[Proposition~3.4, Remark~3.6 and Proposition~7.3]{Garibaldi}.
We first lift a discrete Mather measure to a shift-invariant measure
on compatible bi-infinite configurations. Calibration and
Poincar\'e recurrence then yield static configurations.
The next two lemmas provide the details in our setting.

\begin{lemma}\label{lem:stationary_lifting}
	Let $\mu_\tau$ be a discrete Mather measure and let
	\[
	\Omega_\tau
	:=
	\left\{
	\omega=\{(x_k,v_k)\}_{k\in\mathbb Z}
	\in
	(\mathbb T^d\times\overline B_\kappa)^{\mathbb Z}:
	x_{k+1}=T_\tau(x_k,v_k)\ \ \forall k\in\mathbb Z
	\right\},
	\]
	endowed with the product topology and the shift
	$(\sigma\omega)_k:=\omega_{k+1}$. There exists a
	$\sigma$-invariant Borel probability measure
	$\boldsymbol\mu_\tau$ on $\Omega_\tau$ such that
	\begin{equation}\label{eq:all_transition_marginals}
		(e_k)_\sharp\boldsymbol\mu_\tau=\mu_\tau
		\qquad
		\forall k\in\mathbb Z,
		\qquad\text{where } e_k(\omega):=(x_k,v_k).
	\end{equation}
\end{lemma}

\begin{proof}
Set $S:=X_\kappa$ and let $m_\tau$ be the first marginal of $\mu_\tau$.
  By disintegration,
  $\mu_\tau(dx\,dv)=m_\tau(dx)P_x(dv)$ for a Borel probability kernel
  $P_x$ on $\overline B_\kappa$. Holonomy gives
  $(T_\tau)_\sharp\mu_\tau=m_\tau$. Define
  \[
    R((x,v),\cdot):=\delta_{T_\tau(x,v)}\otimes P_{T_\tau(x,v)}.
  \]
  For every bounded Borel function $f$ on $S$,
  \[
    \int_S Rf\,d\mu_\tau
    =\int_{\mathbb T^d}\int_{\overline B_\kappa}
       f(y,w)\,P_y(dw)\,dm_\tau(y)
    =\int_S f\,d\mu_\tau,
  \]
  so $\mu_\tau$ is $R$-invariant. The Ionescu--Tulcea theorem
  \cite[Theorem~6.17]{Kallenberg}, see also \cite{IonescuTulcea1949},
  constructs the stationary Markov chain with law $\mu_\tau$ and kernel
  $R$. Its stationary finite-dimensional distributions extend by
  Kolmogorov's theorem to a shift-invariant measure on $S^{\mathbb Z}$
  satisfying \eqref{eq:all_transition_marginals}.
  The kernel enforces $x_{k+1}=T_\tau(x_k,v_k)$ almost surely at every
  index. Taking the countable intersection over $k\in\mathbb Z$ shows
  that this measure is supported on the closed set $\Omega_\tau$.
\end{proof}

Each $\omega\in\Omega_\tau$ determines winding numbers
$\ell_k:=\tau v_k-\iota(x_{k+1})+\iota(x_k)\in\mathbb Z^d$.
By \eqref{eq:velocity_action_winding}, the resulting configuration has
velocities $v_k=v_\tau(x_k,x_{k+1},\ell_k)$ and critical costs
\[
c_\tau(x_k,x_{k+1},\ell_k)=\tau L(x_k,v_k)-\tau\bar L(\tau).
\]
 We say
that $\omega$ is \emph{topologically recurrent} if
$\sigma^{q_j}\omega\to\omega$ for some sequence $q_j\to+\infty$.

\begin{lemma}\label{lem:recurrent_calibrated_static}
	Every $\omega\in\Omega_\tau$ that is calibrated by $u_\tau$ and
	topologically recurrent is static.
\end{lemma}

\begin{proof}
	Fix $m<n$. Summing the calibration identities, then using the
	definition \eqref{eq:discrete_mane_potential} of $\Phi_\tau$ and the
	domination \eqref{eq:mane_domination}, we get
	\[
	u_\tau(x_n)-u_\tau(x_m)
	=
	\sum_{k=m}^{n-1}c_\tau(x_k,x_{k+1},\ell_k)
	\geq
	\Phi_\tau(x_m,x_n)
	\geq
	u_\tau(x_n)-u_\tau(x_m).
	\]
	Hence
	\begin{equation}\label{eq:mane_along_calibrated}
		\sum_{k=m}^{n-1}c_\tau(x_k,x_{k+1},\ell_k)
		=
		\Phi_\tau(x_m,x_n)
		=
		u_\tau(x_n)-u_\tau(x_m)
		\qquad\forall m<n,
	\end{equation}
	and in particular $\omega$ is semistatic.
	
	Let $\sigma^{q_j}\omega\to\omega$ with $q_j\to+\infty$. Then
	$x_{m+q_j}\to x_m$. For $j$ large we have $n<m+q_j$, and
	\eqref{eq:mane_along_calibrated} applied to the pair $(n,m+q_j)$
	gives
	\[
	\Phi_\tau(x_n,x_{m+q_j})
	=
	u_\tau(x_{m+q_j})-u_\tau(x_n).
	\]
	By continuity of $\Phi_\tau$ (\Cref{lem:mane_potential}) and of
	$u_\tau$, letting $j\to\infty$ yields
	\[
	\Phi_\tau(x_n,x_m)
	=
	u_\tau(x_m)-u_\tau(x_n)
	=
	-\Phi_\tau(x_m,x_n).
	\]
	This is \eqref{eq:discrete_static}.
\end{proof}

The next proposition implements the strategy recalled above.
By \Cref{lem:stationary_lifting}, a discrete Mather measure lifts to a
shift-invariant measure concentrated on calibrated configurations.
Poincar\'e recurrence and \Cref{lem:recurrent_calibrated_static} then
show that almost every such configuration is static, yielding the
inclusion of the discrete Mather set in the discrete Aubry set.

\begin{proposition}\label{prop:discrete_set_inclusions}
	For every $\tau\in(0,1)$, the discrete Mather, Aubry and Ma\~n\'e sets
	are nonempty compact subsets of
	$\mathbb T^d\times\overline B_\kappa$, and
	\begin{equation}\label{eq:discrete_set_inclusions}
		\widetilde{\mathcal M}_L^\tau
		\subset
		\widetilde{\mathcal A}_L^\tau
		\subset
		\widetilde{\mathcal N}_L^\tau.
	\end{equation}
\end{proposition}

\begin{proof}
	The compactness of $\widetilde{\mathcal N}_L^\tau$ and
	$\widetilde{\mathcal A}_L^\tau$ and the second inclusion are
	\Cref{lem:closedness}.	
	Let $\mu_\tau$ be a discrete Mather measure and let
	$\boldsymbol\mu_\tau$ be given by \Cref{lem:stationary_lifting}. Set
	\[
	\Omega^{\rm cal}
	:=
	\bigcap_{k\in\mathbb Z}e_k^{-1}\bigl(\operatorname{spt}(\mu_\tau)\bigr).
	\]
	By \eqref{eq:all_transition_marginals},
	$\boldsymbol\mu_\tau(\Omega^{\rm cal})=1$, and by
	\eqref{eq:zero_defect_support} every $\omega\in\Omega^{\rm cal}$ is
	calibrated by $u_\tau$.	The space $\Omega_\tau$ is compact and metrizable, and $\sigma$ is a
	homeomorphism of $\Omega_\tau$ preserving $\boldsymbol\mu_\tau$. Let
	$\{U_i\}_{i\in\mathbb N}$ be a countable basis of $\Omega_\tau$. By
	the Poincar\'e recurrence theorem, the sets
	\[
	Z_i:=U_i\setminus\bigcap_{N\geq1}\bigcup_{q\geq N}\sigma^{-q}(U_i)
	\]
	are $\boldsymbol\mu_\tau$-null. Hence
	$\Omega^{\rm rec}:=\Omega_\tau\setminus\bigcup_iZ_i$ has full measure and 
	every point of $\Omega^{\rm rec}$ is topologically recurrent.
	
	By \Cref{lem:recurrent_calibrated_static}, every
	$\omega\in\Omega^{\rm cal}\cap\Omega^{\rm rec}$ is static, so
	$e_0(\omega)\in\widetilde{\mathcal A}_L^\tau$ by
	\Cref{def:discrete_mane_aubry_sets}. Therefore
	\[
	\mu_\tau\bigl(\widetilde{\mathcal A}_L^\tau\bigr)
	=
	\boldsymbol\mu_\tau\bigl(e_0^{-1}(\widetilde{\mathcal A}_L^\tau)\bigr)
	\geq
	\boldsymbol\mu_\tau\bigl(\Omega^{\rm cal}\cap\Omega^{\rm rec}\bigr)
	=1.
	\]
	Since $\widetilde{\mathcal A}_L^\tau$ is closed, it contains
	$\operatorname{spt}(\mu_\tau)$. As $\mu_\tau$ is arbitrary, the
	closure of the union of these supports is also contained in
	$\widetilde{\mathcal A}_L^\tau$. This proves
	$\widetilde{\mathcal M}_L^\tau\subset\widetilde{\mathcal A}_L^\tau$.
	
	Discrete Mather measures exist by \Cref{lem:minimum_action}, and their
	supports are nonempty. Hence $\widetilde{\mathcal M}_L^\tau$ is
	nonempty, and so are the other two sets by
	\eqref{eq:discrete_set_inclusions}. Finally,
	$\widetilde{\mathcal M}_L^\tau$ is a closed subset of the compact set
	$\widetilde{\mathcal A}_L^\tau$, hence compact.
\end{proof}

\section{Convergence of the Ma\~n\'e set}\label{sec:conv_mane}

We begin with two elementary structural properties of discrete semistatic configurations that will be used to obtain compactness as the time step tends to zero.
Semistaticity and staticity are preserved under shifts of the discrete
time index. Consequently, every phase-space point of a semistatic
configuration belongs to the discrete Ma\~n\'e set, and every
phase-space point of a static configuration belongs to the discrete
Aubry set.

\begin{lemma}\label{lem:shift_invariance}
  Let $\{(x_k,\ell_k)\}_{k\in\mathbb Z}$ be a semistatic (respectively,
  static) configuration for the step $\tau$, and let $m\in\mathbb Z$. Then
  the shifted configuration $\{(x_{k+m},\ell_{k+m})\}_{k\in\mathbb Z}$ is
  semistatic (respectively, static). Consequently, with
  $v_k:=v_\tau(x_k,x_{k+1},\ell_k)$,
  \begin{equation}\label{eq:shift_invariance}
    (x_k,v_k)\in\widetilde{\mathcal N}_L^\tau
    \quad\bigl(\text{respectively, }(x_k,v_k)\in\widetilde{\mathcal A}_L^\tau\bigr)
    \qquad\text{for every }k\in\mathbb Z.
  \end{equation}
\end{lemma}

\begin{proof}
Apply \eqref{eq:discrete_semistatic} and, in the static case,
  \eqref{eq:discrete_static} to $(p+m,q+m)$ and reindex the sum.
  These are precisely the defining identities for the shifted
  configuration. Its initial phase-space point is $(x_m,v_m)$, so
  \eqref{eq:shift_invariance} follows from
  \Cref{def:discrete_mane_aubry_sets}.
\end{proof}
 To compare discrete configurations with continuous
trajectories and to derive the subsequent Euler--Lagrange estimates, we now
represent their successive displacements in the universal cover. Given a configuration
\(\{(x_k,\ell_k)\}_{k\in\mathbb Z}\), set
\(v_k:=v_\tau(x_k,x_{k+1};\ell_k)\), fix
\(\widetilde x_0\in\pi^{-1}(x_0)\), and define recursively
\begin{equation}\label{eq:compatible_lifts}
	\widetilde x_{k+1}:=\widetilde x_k+\tau v_k,
	\qquad
	\widetilde x_{k-1}:=\widetilde x_k-\tau v_{k-1},
	\qquad k\in\mathbb Z.
\end{equation}
Since \(\tau v_k=\Delta_{\ell_k}(x_k,x_{k+1})\) is a lift of
\(x_{k+1}-x_k\), induction gives
\(\pi(\widetilde x_k)=x_k\) for every \(k\in\mathbb Z\).
Recall that \(\kappa\) is the velocity bound from
\Cref{lem:velocity_bound}. Since
\(D^2_{vv}L=(D^2_{pp}H)^{-1}\) is positive definite and continuous, there
exists \(\theta>0\) such that
\begin{equation}\label{eq:uniform_convexity_mane}
	D^2_{vv}L(x,w)\geq\theta\,\mathrm{Id}
	\qquad\text{for every }(x,w)\in X_\kappa.
\end{equation}
We also define
\begin{equation}
	M_1:=\max_{X_\kappa}|D_xL|,
	\qquad
	M_2:=\max_{X_\kappa}\|D^2_{vx}L\|,
	\qquad
	C_{\rm EL}:=\frac{\kappa M_2+M_1}{\theta},
\end{equation}
where \(D^2_{vx}L\) denotes the matrix of mixed derivatives
\(\partial^2L/\partial v_i\partial x_j\).  
The variational minimality of semistatic configurations yields a discrete
Euler--Lagrange equation which, together with the uniform convexity of \(L\)
in the velocity variable, provides a uniform estimate on the velocity
increments.
\begin{lemma} 
  Let $\{(x_k,\ell_k)\}_{k\in\mathbb Z}$ be a semistatic configuration and
  $v_k:=v_\tau(x_k,x_{k+1},\ell_k)$. Then
  \begin{equation}\label{eq:discrete_EL_semistatic}
    D_vL(x_k,v_k)-D_vL(x_{k+1},v_{k+1})+\tau D_xL(x_{k+1},v_{k+1})=0
    \qquad\text{for every }k\in\mathbb Z.
  \end{equation}
  Consequently,
  \begin{equation}\label{eq:semistatic_velocity_increment}
    |v_{k+1}-v_k|\leq C_{\rm EL}\,\tau
    \qquad\text{for every }k\in\mathbb Z.
  \end{equation}
\end{lemma}

\begin{proof}
Fix $k$ and compatible lifts \eqref{eq:compatible_lifts}. Set
  \[
    S_k(y):=\tau L\Bigl(x_k,\frac{y-\widetilde x_k}{\tau}\Bigr)
      +\tau L\Bigl(\pi(y),\frac{\widetilde x_{k+2}-y}{\tau}\Bigr),
      \qquad y\in\mathbb R^d.
  \]
  Each $y$ defines an admissible two-step chain from $x_k$ to $x_{k+2}$,
  since its two lifted displacements determine valid winding numbers.
  Semistaticity therefore gives
  \[
    S_k(\widetilde x_{k+1})
    =\Phi_\tau(x_k,x_{k+2})+2\tau\bar L(\tau)
    \leq S_k(y).
  \]
  Differentiating at the minimum yields
  \eqref{eq:discrete_EL_semistatic}, which we rewrite as
  \begin{equation}\label{eq:EL_rewritten}
    D_vL(x_{k+1},v_{k+1})-D_vL(x_{k+1},v_k)
    =D_vL(x_k,v_k)-D_vL(x_{k+1},v_k)
      +\tau D_xL(x_{k+1},v_{k+1}).
  \end{equation}
  The mean-value estimate along the lifted segment of length
  $\tau|v_k|$ bounds the right-hand side by
  $(\kappa M_2+M_1)\tau$. On the left, uniform convexity
  \eqref{eq:uniform_convexity_mane} gives, with
  $\delta_k=v_{k+1}-v_k$,
  \[
    \bigl[D_vL(x_{k+1},v_{k+1})-D_vL(x_{k+1},v_k)\bigr]\cdot\delta_k
    \geq\theta|\delta_k|^2.
  \]
  The velocity arguments lie in $\overline B_\kappa$ by
  \Cref{lem:velocity_bound}. Cauchy--Schwarz and
  \eqref{eq:EL_rewritten} now imply
  $\theta|\delta_k|\leq(\kappa M_2+M_1)\tau$, proving
  \eqref{eq:semistatic_velocity_increment}.
\end{proof}

The uniform velocity bound of \Cref{lem:velocity_bound} and the increment estimate \eqref{eq:semistatic_velocity_increment} now give the compactness needed to pass from discrete semistatic configurations to continuous Euler-Lagrange trajectories.

\begin{lemma} \label{lem:compactness_semistatic_configurations}
  Let $\tau_i\to0$ and let $\{(x_k^i,\ell_k^i)\}_{k\in\mathbb Z}$ be
  semistatic configurations for the step $\tau_i$, with velocities
  $v_k^i:=v_{\tau_i}(x_k^i,x_{k+1}^i,\ell_k^i)$, such that
  \begin{equation}
    (x_0^i,v_0^i)\longrightarrow(x,v)\qquad\text{in }X.
  \end{equation}
  Then there exists a $C^2$
  Euler--Lagrange trajectory $\gamma:\mathbb R\to\mathbb T^d$ with
  \begin{equation}\label{eq:limit_initial_conditions}
    \bigl(\gamma(t),\dot\gamma(t)\bigr)=\phi^L_t(x,v)
    \qquad\text{for every }t\in\mathbb R,
  \end{equation}
  such that, with $k_i(t):=\lfloor t/\tau_i\rfloor$,
  \begin{equation}\label{eq:grid_convergence}
    \sup_{|t|\leq M}
    \Bigl[d_{\mathbb T^d}\bigl(x_{k_i(t)}^i,\gamma(t)\bigr)
    +\bigl|v_{k_i(t)}^i-\dot\gamma(t)\bigr|\Bigr]\longrightarrow0
    \qquad\text{for every }M>0.
  \end{equation}
  In particular $|\dot\gamma(t)|\leq\kappa$ for every $t\in\mathbb R$.
\end{lemma}

\begin{proof}
Since the limit trajectory identified below is $\gamma_{x,v}$ and does not
depend on the extraction, it suffices to show that every subsequence has a
further subsequence along which \eqref{eq:grid_convergence} holds with
$\gamma=\gamma_{x,v}$; extractions below are of this kind.
Choose compatible lifts with $\widetilde x_0^i\to\widetilde x$.
  Let $\widetilde\gamma_i$ and $\widehat V_i$ be the piecewise affine
  interpolations of $\widetilde x_k^i$ and $v_k^i$ at times $k\tau_i$.
  By \Cref{lem:velocity_bound} and
  \eqref{eq:semistatic_velocity_increment}, $\widehat V_i$ takes values
  in $\overline B_\kappa$ and is $C_{\rm EL}$-Lipschitz. Arzel\`a--Ascoli
  and a diagonal extraction give locally uniform convergence
  $\widehat V_i\to V$. Moreover,
  $|\dot{\widetilde\gamma}_i-\widehat V_i|\leq C_{\rm EL}\tau_i$
  almost everywhere. Hence, with
  $\widetilde\gamma(t):=\widetilde x+\int_0^tV(s)\,ds$,
  \[
    \sup_{|t|\leq M}|\widetilde\gamma_i(t)-\widetilde\gamma(t)|
    \leq|\widetilde x_0^i-\widetilde x|
      +M\left(C_{\rm EL}\tau_i+\sup_{[-M,M]}|\widehat V_i-V|\right)
    \longrightarrow0.
  \]
  Thus $\gamma=\pi\circ\widetilde\gamma$ is $C^1$, with
  $\dot\gamma=V$, $(\gamma(0),\dot\gamma(0))=(x,v)$ and
  $|\dot\gamma|\leq\kappa$. On each grid interval, the estimates
  $|\widetilde x_{k_i(t)}^i-\widetilde\gamma_i(t)|\leq\kappa\tau_i$ and
  $|v_{k_i(t)}^i-\widehat V_i(t)|\leq C_{\rm EL}\tau_i$ imply
  \eqref{eq:grid_convergence}.

  Set $p_k^i:=D_vL(x_k^i,v_k^i)$ and let $\widehat P_i$ be its affine
  interpolation. The discrete Euler--Lagrange equation gives
  \begin{equation}\label{eq:discrete_momentum}
    \dot{\widehat P}_i(t)=D_xL(x_{k+1}^i,v_{k+1}^i)
    \quad\text{for }t\in(k\tau_i,(k+1)\tau_i).
  \end{equation}
  In particular $|\widehat P_i(t)-p_{k_i(t)}^i|\leq M_1\tau_i$.
  By \eqref{eq:grid_convergence} and uniform continuity on $X_\kappa$,
  $\widehat P_i\to p:=D_vL(\gamma,\dot\gamma)$ locally uniformly,
  while the right-hand side of \eqref{eq:discrete_momentum} converges
  locally uniformly to $D_xL(\gamma,\dot\gamma)$, using
  $(k_i(t)+1)\tau_i-t\in(0,\tau_i]$.
  Passing to the limit in the integrated momentum equation yields
  \[
    p(t)-p(0)=\int_0^tD_xL(\gamma(s),\dot\gamma(s))\,ds.
  \]
  Hence $p$ is $C^1$ and $(\gamma,p)$ solves
  $\dot\gamma=D_pH(\gamma,p)$, $\dot p=-D_xH(\gamma,p)$.
  Since this Hamiltonian vector field is $C^1$, $\gamma$ is $C^2$,
  and uniqueness with initial datum $(x,v)$ gives
  \eqref{eq:limit_initial_conditions}.
\end{proof}

\subsection{Upper Kuratowski limit}

The proof of the upper limit relies on
\Cref{lem:compactness_semistatic_configurations} and on a comparison between
the discrete and continuous Ma\~n\'e potentials, which is the content of the
next two lemmas.

\begin{lemma}\label{lem:phi_quasi_lip}
  For every $\tau\in(0,1)$ and every $x,y\in\mathbb T^d$,
  \begin{equation}\label{eq:phi_quasi_lipschitz}
    \Phi_\tau(x,y)\leq(A+C_0)\bigl(d_{\mathbb T^d}(x,y)+\tau\bigr),
  \end{equation}
  where $A$ and $C_0$ are the constants in \eqref{eq:standing_constants}.
  In particular, if $\tau_i\to0$ and $d_{\mathbb T^d}(x_i,y_i)\to0$, then
  $\displaystyle{\limsup_i} \; \Phi_{\tau_i}(x_i,y_i)\leq0$.
\end{lemma}

\begin{proof}
  Let $\delta:=d_{\mathbb T^d}(x,y)$. Choose $w\in\pi^{-1}(y-x)$ with
  $|w|=\delta$ (the distance on $\mathbb T^d$ is realized by a shortest
  representative), a lift $\widehat x$ of $x$, and set
  $n:=\max\{1,\lceil\delta/\tau\rceil\}$. Consider the chain
  \[
    z_j:=\pi\Bigl(\widehat x+\frac jnw\Bigr),\qquad j=0,\dots,n,
  \]
  which satisfies $z_0=x$ and $z_n=y$. For every $j$, the vector $w/n$ is a
  lift of $z_{j+1}-z_j$, so there is a winding number $\ell_j$ with
  $\Delta_{\ell_j}(z_j,z_{j+1})=w/n$, and the corresponding velocity is
  $w/(n\tau)$. Since $n\geq\delta/\tau$, we have $|w/(n\tau)|\leq1$, hence
  $L(z_j,w/(n\tau))\leq A$ by \eqref{eq:standing_constants}. Moreover
  $-\bar L(\tau)\leq C_0$ by \eqref{eq:ergodic_constant_bounds}. Using this
  chain in the definition of $\Phi_\tau$, we get
  \[
    \Phi_\tau(x,y)\leq\sum_{j=0}^{n-1}\tau\Bigl[L\Bigl(z_j,\frac{w}{n\tau}\Bigr)
    -\bar L(\tau)\Bigr]\leq(A+C_0)\,n\tau.
  \]
  Finally, if $\delta\leq\tau$ then $n=1$ and $n\tau=\tau$, while if
  $\delta>\tau$ then $n\tau<\delta+\tau$. In both cases
  $n\tau\leq\delta+\tau$, which gives \eqref{eq:phi_quasi_lipschitz}. The
  last assertion follows immediately.
\end{proof}

\begin{lemma}\label{lem:discretization}
  Let $\xi:[0,T]\to\mathbb T^d$ be Lipschitz. For $\tau\in(0,1)$ set
  $N_\tau:=\lceil T/\tau\rceil$, $h_\tau:=T/N_\tau$ and
  $z_j^\tau:=\xi(jh_\tau)$, $j=0,\dots,N_\tau$. Then
  \begin{equation}\label{eq:discretized_action}
    \limsup_{\tau\to0}\sum_{j=0}^{N_\tau-1}
    \hat c_\tau(z_j^\tau,z_{j+1}^\tau)
    \leq\int_0^T\bigl(L(\xi(t),\dot\xi(t))+\alpha(H)\bigr)\,dt.
  \end{equation}
\end{lemma}

\begin{proof}
  Write, for simplicity of notation, $N=N_\tau$ and $h=h_\tau$. Since $T/\tau \leq N<T/\tau+1$ we have $h\leq\tau$ and 
  \begin{equation}\label{eq:N_tau_bound}
    0\leq N\tau-T=N(\tau-h)<\tau.
  \end{equation}
  Let $\widetilde\xi$ be a lift of $\xi$ and $R$ its Lipschitz constant, so
  that $|\dot{\widetilde\xi}|\leq R$ a.e. Set
  \[
    \bar v_j:=\frac1h\int_{jh}^{(j+1)h}\dot{\widetilde\xi}(s)\,ds,
    \qquad
    w_j:=\frac{\widetilde\xi((j+1)h)-\widetilde\xi(jh)}{\tau}
    =\frac h\tau\,\bar v_j.
  \]
  Then $|\bar v_j|,|w_j|\leq R$ and
  $|w_j-\bar v_j|=(1-h/\tau)|\bar v_j|\leq R(\tau-h)/\tau$. Let $M_R$ be a
  bound for $|L|$ on $X_R$, $\Lambda_R$ a Lipschitz constant of $L$ on $X_R$
  (which exists since $L\in C^2$), and $\omega_R$ a nondecreasing modulus of
  continuity of $L$ on $X_R$.

  Fix $j$. Since $\tau w_j$ is a lift of $z_{j+1}^\tau-z_j^\tau$,
  \Cref{lem:reduced_cost} gives
  \[
  \hat c_\tau(z_j^\tau,z_{j+1}^\tau)+\tau\bar L(\tau)
  =\widehat{\mathcal L}_\tau(z_j^\tau,z_{j+1}^\tau)\leq\tau L(z_j^\tau,w_j).
  \]
  Next,
  \begin{align*}
    \tau L(z_j^\tau,w_j)&=hL(z_j^\tau,w_j)+(\tau-h)L(z_j^\tau,w_j)\\
    &\leq hL(z_j^\tau,\bar v_j)+h\Lambda_R|w_j-\bar v_j|+(\tau-h)M_R
    \leq hL(z_j^\tau,\bar v_j)+(\Lambda_RR+M_R)(\tau-h),
  \end{align*}
  where we used $h/\tau\leq1$. By the convexity of $L(z_j^\tau,\cdot)$ and
  Jensen's inequality, we have that
  \[
    hL(z_j^\tau,\bar v_j)\leq\int_{jh}^{(j+1)h}L\bigl(z_j^\tau,\dot\xi(s)\bigr)\,ds.
  \]
  For $s\in[jh,(j+1)h]$ we have
  $d_{\mathbb T^d}(z_j^\tau,\xi(s))\leq Rh\leq R\tau$, and therefore
  \[
  L(z_j^\tau,\dot\xi(s))\leq L(\xi(s),\dot\xi(s))+\omega_R(R\tau).
  \]
  Collecting
  these inequalities,
  \[
    \hat c_\tau(z_j^\tau,z_{j+1}^\tau)
    \leq\int_{jh}^{(j+1)h}L\bigl(\xi(s),\dot\xi(s)\bigr)\,ds
    +h\,\omega_R(R\tau)+(\Lambda_RR+M_R)(\tau-h)-\tau\bar L(\tau).
  \]
  Summing over $j=0,\dots,N-1$ and using \eqref{eq:N_tau_bound},
  \[
    \sum_{j=0}^{N-1}\hat c_\tau(z_j^\tau,z_{j+1}^\tau)
    \leq\int_0^TL(\xi,\dot\xi)\,dt+T\omega_R(R\tau)+(\Lambda_RR+M_R)\tau
      -N\tau\,\bar L(\tau).
  \]
  As $\tau\to0$, $\omega_R(R\tau)\to0$, $N\tau\to T$ by
  \eqref{eq:N_tau_bound}, and $\bar L(\tau)\to-\alpha(H)$ by
  \eqref{conv_ergodic_const}. Hence the upper limit of the right-hand side is
  $\int_0^TL(\xi,\dot\xi)\,dt+\alpha(H)T$, which is
  \eqref{eq:discretized_action}.
\end{proof}


\begin{proposition}\label{prop:limit_is_semistatic}
  Under the assumptions and with the convergences of
  \Cref{lem:compactness_semistatic_configurations}, the limiting
  Euler--Lagrange trajectory $\gamma$ is semistatic. Consequently,
  $(x,v)\in\widetilde{\mathcal N}_L$.
\end{proposition}

\begin{proof}
  Fix $a<b$ and set $m_i:=\lfloor a/\tau_i\rfloor$ and
  $n_i:=\lfloor b/\tau_i\rfloor$. Since $n_i-m_i>(b-a)/\tau_i-1$, we have
  $m_i<n_i$ for $i$ large, and $m_i\tau_i\to a$, $n_i\tau_i\to b$.

  By
  \eqref{eq:grid_convergence}, $x_{m_i}^i\to\gamma(a)$ and
  $x_{n_i}^i\to\gamma(b)$. By discrete semistaticity
  \eqref{eq:discrete_semistatic},
  \begin{equation}\label{eq:discrete_action_identity}
    \Phi_{\tau_i}(x_{m_i}^i,x_{n_i}^i)
    =\sum_{k=m_i}^{n_i-1}\tau_i\bigl[L(x_k^i,v_k^i)-\bar L(\tau_i)\bigr]
    =\int_{m_i\tau_i}^{n_i\tau_i}
      \bigl[L\bigl(x_{k_i(s)}^i,v_{k_i(s)}^i\bigr)-\bar L(\tau_i)\bigr]\,ds,
  \end{equation}
  because $k_i(s)=k$ for $s\in[k\tau_i,(k+1)\tau_i)$. All the points
  $(x_k^i,v_k^i)$ and $(\gamma(s),\dot\gamma(s))$ lie in $X_\kappa$, where $L$
  is uniformly continuous and bounded. By \eqref{eq:grid_convergence} on
  $[a-1,b+1]$, the integrand in \eqref{eq:discrete_action_identity}
  converges to $L(\gamma(s),\dot\gamma(s))+\alpha(H)$ uniformly on this
  interval (recall \eqref{conv_ergodic_const}); since the integration
  intervals converge to $[a,b]$ and the integrands are uniformly bounded,
  \begin{equation}\label{eq:discrete_action_limit}
    \Phi_{\tau_i}(x_{m_i}^i,x_{n_i}^i)\longrightarrow
    \int_a^b\bigl(L(\gamma(s),\dot\gamma(s))+\alpha(H)\bigr)\,ds.
  \end{equation}

  Let
  $\varepsilon>0$. By definition of $\Phi_L$, there is $T>0$ with
  $h_T(\gamma(a),\gamma(b))+\alpha(H)T\leq\Phi_L(\gamma(a),\gamma(b))+\varepsilon$.
  By Tonelli's theorem, the infimum defining $h_T(\gamma(a),\gamma(b))$ is
  attained by a curve of class $C^2$, which is in particular
  Lipschitz. Hence there exists a Lipschitz curve
  $\xi:[0,T]\to\mathbb T^d$ from $\gamma(a)$ to $\gamma(b)$ such that
  \[
    \int_0^T\bigl(L(\xi,\dot\xi)+\alpha(H)\bigr)\,dt
    \leq\Phi_L\bigl(\gamma(a),\gamma(b)\bigr)+\varepsilon.
  \]
  Let $z_j^{\tau_i}$, $j=0,\dots,N_{\tau_i}$, be the chain of
  \Cref{lem:discretization}. Choosing for each transition an optimal winding
  number, this chain is admissible in the definition of
  $\Phi_{\tau_i}(\gamma(a),\gamma(b))$, so that
  \[
    \Phi_{\tau_i}\bigl(\gamma(a),\gamma(b)\bigr)
    \leq\sum_{j=0}^{N_{\tau_i}-1}\hat c_{\tau_i}(z_j^{\tau_i},z_{j+1}^{\tau_i}).
  \]
  By the triangle inequality \eqref{eq:mane_triangle},
  \[
    \Phi_{\tau_i}(x_{m_i}^i,x_{n_i}^i)
    \leq\Phi_{\tau_i}\bigl(x_{m_i}^i,\gamma(a)\bigr)
    +\sum_{j=0}^{N_{\tau_i}-1}\hat c_{\tau_i}(z_j^{\tau_i},z_{j+1}^{\tau_i})
    +\Phi_{\tau_i}\bigl(\gamma(b),x_{n_i}^i\bigr).
  \]
  The first and last terms have nonpositive upper limit by
  \Cref{lem:phi_quasi_lip}, since $x_{m_i}^i\to\gamma(a)$ and
  $x_{n_i}^i\to\gamma(b)$. The upper limit of the middle term is bounded by
  $\Phi_L(\gamma(a),\gamma(b))+\varepsilon$ by \eqref{eq:discretized_action}.
  Together with \eqref{eq:discrete_action_limit}, this gives
  \[
    \int_a^b\bigl(L(\gamma,\dot\gamma)+\alpha(H)\bigr)\,ds
    \leq\Phi_L\bigl(\gamma(a),\gamma(b)\bigr)+\varepsilon.
  \]
  Since $\varepsilon>0$ is arbitrary, the left-hand side is bounded by
  $\Phi_L(\gamma(a),\gamma(b))$.

Since $L$ is autonomous, the curve
  $s\mapsto\gamma(a+s)$, $s\in[0,b-a]$, is admissible in the definition of
  $h_{b-a}(\gamma(a),\gamma(b))$. Hence
  \[
    \Phi_L\bigl(\gamma(a),\gamma(b)\bigr)
    \leq h_{b-a}\bigl(\gamma(a),\gamma(b)\bigr)+\alpha(H)(b-a)
    \leq\int_a^b\bigl(L(\gamma,\dot\gamma)+\alpha(H)\bigr)\,ds.
  \]
  Thus \eqref{eq:semistatic} holds for every $a<b$. Since $\gamma$ is an
  Euler--Lagrange trajectory with $(\gamma(0),\dot\gamma(0))=(x,v)$, that is
  $\gamma=\gamma_{x,v}$, we conclude that $(x,v)\in\widetilde{\mathcal N}_L$.
\end{proof}

\begin{theorem} \label{thm:discrete_mane_limsup}
  The discrete Ma\~n\'e sets satisfy
  \begin{equation}\label{eq:discrete_mane_limsup}
    \limsup_{\tau\to0}\widetilde{\mathcal N}_L^\tau
    \subset\widetilde{\mathcal N}_L.
  \end{equation}
  Moreover,
  \begin{equation}\label{eq:uniform_distance_mane}
    \lim_{\tau\to0}\ \sup_{a\in\widetilde{\mathcal N}_L^\tau}
    \dist_X\bigl(a,\widetilde{\mathcal N}_L\bigr)=0.
  \end{equation}
\end{theorem}

\begin{proof}
  Let $z\in\limsup_{\tau\to0}\widetilde{\mathcal N}_L^\tau$: there are
  $\tau_i\downarrow0$ and $(x_i,v_i)\in\widetilde{\mathcal N}_L^{\tau_i}$ with
  $(x_i,v_i)\to z=(x,v)$. By \Cref{def:discrete_mane_aubry_sets}, for every
  $i$ there is a semistatic configuration
  $\{(x_k^i,\ell_k^i)\}_{k\in\mathbb Z}$ for the step $\tau_i$ with
  $x_0^i=x_i$ and $v_{\tau_i}(x_0^i,x_1^i,\ell_0^i)=v_i$. These configurations
  satisfy the assumptions of
  \Cref{lem:compactness_semistatic_configurations}, and
  \Cref{prop:limit_is_semistatic} gives $(x,v)\in\widetilde{\mathcal N}_L$.
  This proves \eqref{eq:discrete_mane_limsup}.

  The sets $\widetilde{\mathcal N}_L^\tau$ are nonempty, compact and
  contained in the compact set $X_\kappa$ for every $\tau\in(0,1)$
  (\Cref{prop:discrete_set_inclusions}), and $\widetilde{\mathcal N}_L$ is
  nonempty and compact (\Cref{prop:continuous_sets}). Hence
  \eqref{eq:uniform_distance_mane} follows from
  \eqref{eq:discrete_mane_limsup} and
  \Cref{lem:kuratowski_hausdorff}(3) with $\mathcal W=X_\kappa$.
\end{proof}
\begin{remark}
	Theorem~\ref{thm:discrete_mane_limsup} gives only the upper inclusion:
	limits of discrete semistatic configurations are continuous semistatic
	trajectories. Full convergence would additionally require
	\begin{equation*}
		\lim_{\tau\to0}\dist_X\bigl(z,\widetilde{\mathcal N}_L^\tau\bigr)=0
		\qquad\text{for every }z\in\widetilde{\mathcal N}_L.
	\end{equation*}
	This amounts to approximating each continuous semistatic trajectory
	by bi-infinite discrete semistatic configurations. Consistency of the
	one-step action and convergence of critical values and weak KAM
	solutions do not by themselves provide this global recovery property.
	At present, we do not establish it in general.
	
	Full convergence is obtained for a single Aubry orbit in
	\Cref{thm:single_component}, and under uniqueness of the Mather measure
	with support equal to the Ma\~n\'e set in
	\Cref{cor:full_convergence_unique_mather}.
\end{remark}

\section{Convergence of the Aubry set}\label{sec:conv_aubry}

We study the behavior of $\widetilde{\mathcal A}_L^\tau$ as $\tau\to0$.
Compactness alone gives semistatic limits, but need not preserve staticity.
In \Cref{subsec:limsup_aubry} we prove the upper inclusion when the
continuous Aubry set has finitely many static classes. In \Cref{subsec:liminf_aubry} we
obtain full convergence for a single equilibrium or periodic orbit, and
recover equilibrium components under a quadratic critical gap.

\subsection{Upper Kuratowski limit}\label{subsec:limsup_aubry}

We adapt the level-separation mechanism underlying Bernard's
semicontinuity theorem \cite[Section~3]{BernardConley2010}. His theorem
concerns perturbations of continuous Tonelli Lagrangians; the argument
below treats the present semi-discrete action directly, using circuits
with small calibration defect.

Two points $x,y\in\mathcal A_L$ belong to the same \emph{static class}
if $\Phi_L(x,y)+\Phi_L(y,x)=0$. These equivalence classes form the
quotient Aubry set. If $u_1,u_2$ are critical subsolutions, domination in
both directions gives
\[
  u_i(y)-u_i(x)=\Phi_L(x,y),\qquad i=1,2,
\]
whenever $x,y$ belong to the same class. Thus $u_1-u_2$ is constant on
each static class \cite{Fathi2008,BernardConley2010}.

\begin{assumption} 
\label{ass:finite_static_classes}
The continuous Aubry set has finitely many static classes.
\end{assumption}

This assumption holds, in particular, if $\widetilde{\mathcal A}_L$ is a
finite union of equilibria and periodic orbits: the projection of each
such static orbit lies in a single static class.

By \cite[Theorem~B]{BernardC112007}, there exists a critical subsolution
$u^\sharp\in C^{1,1}(\mathbb T^d)$ such that
\[
  a(x):=\alpha(H)-H(x,Du^\sharp(x))\geq0,
  \qquad a(x)>0\quad\text{for }x\notin\mathcal A_L.
\]
Fix this function and set
\begin{equation*}
  Q(x,v):=L(x,v)+\alpha(H)-Du^\sharp(x)\cdot v,
  \qquad w_\tau:=u^\sharp-u_\tau.
\end{equation*}
Fenchel's inequality gives $Q(x,v)\geq a(x)$. Moreover,
$Q^{-1}(0)=\widetilde{\mathcal A}_L$: equality requires
$x\in\mathcal A_L$ and $v=D_pH(x,Du^\sharp(x))$, which is precisely
the Aubry graph under the Legendre transform.

For a finite chain with transitions $(x_k,v_k)$, $k=m,\ldots,n-1$,
its \emph{total calibration defect} is
$\sum_{k=m}^{n-1}G_\tau(x_k,v_k)$.

\begin{lemma}\label{lem:static_small_defect_circuits}
	Let $\tau\in(0,1)$ and consider a finite segment of a discrete static
	configuration, running from $x_m$ to $x_n$, with $m<n$.
	For every $\varepsilon>0$, one can append a finite return chain
	from $x_n$ to $x_m$ so that the resulting closed chain has total
	calibration defect at most $\varepsilon$.
	
	Moreover, there exist constants $R,C>0$, independent of $\tau$
	and of the length of the chain, such that every transition
	$(x_k,v_k)$ of any finite chain with total calibration defect
	at most $\tau^2$ satisfies
	\begin{align}
		|v_k|&\leq R,\qquad
		|w_\tau(x_{k+1})-w_\tau(x_k)|\leq C\tau,
		\label{eq:small_defect_increment_bound}\\
		w_\tau(x_{k+1})-w_\tau(x_k)
		&\leq-\tau Q(x_k,v_k)+C\tau^2.
		\label{eq:small_defect_descent}
	\end{align}
\end{lemma}

\begin{proof}
Let the static segment run from $x_m$ to $x_n$, $m<n$. Domination and
staticity imply
$u_\tau(x_n)-u_\tau(x_m)=\Phi_\tau(x_m,x_n)$, so its defect is zero.
Choose a return chain from $x_n$ to $x_m$ with critical action at most
$\Phi_\tau(x_n,x_m)+\varepsilon$. On the resulting closed chain the
increments of $u_\tau$ telescope. Hence its total defect equals its
critical action and belongs to $[0,\varepsilon]$.

For a chain with total defect at most $\tau^2$, nonnegativity gives
$G_\tau(x_k,v_k)\leq\tau^2$ for every transition. The uniform Lipschitz
bound on $u_\tau$, \eqref{eq:ergodic_constant_bounds}, and compatibility
give
\[
  L(x_k,v_k)
  \leq\bar L(\tau)
    +\frac{u_\tau(x_{k+1})-u_\tau(x_k)}{\tau}+\tau
  \leq A+K_0|v_k|+1.
\]
Superlinearity therefore yields a uniform bound $|v_k|\leq R$.
Since $w_\tau$ has uniformly bounded Lipschitz constant and
$d_{\mathbb T^d}(x_k,x_{k+1})\leq\tau|v_k|$, this proves
\eqref{eq:small_defect_increment_bound}. Taylor's estimate for the
periodic lift of $u^\sharp$ gives
\[
  u^\sharp(x_{k+1})-u^\sharp(x_k)
  \leq\tau Du^\sharp(x_k)\cdot v_k
    +\tfrac12\operatorname{Lip}(Du^\sharp)R^2\tau^2.
\]
Subtracting the defining identity for $G_\tau(x_k,v_k)$ yields
\[
  w_\tau(x_{k+1})-w_\tau(x_k)
  \leq-\tau Q(x_k,v_k)+\tau(\bar L(\tau)+\alpha(H))
    +G_\tau(x_k,v_k)
    +\tfrac12\operatorname{Lip}(Du^\sharp)R^2\tau^2.
\]
Now \eqref{eq:critical_value_rate} and $G_\tau(x_k,v_k)\leq\tau^2$
give \eqref{eq:small_defect_descent}.
\end{proof}

\begin{lemma}\label{lem:no_upward_crossing}
	Let $\tau_n\to0$ and assume that $u_{\tau_n}\to u$ uniformly.
	Set $w:=u^\sharp-u$, and choose $s\in\mathbb R$ such that
	$w(x)\neq s$ for every $x\in\mathcal A_L$.

	Then, for all sufficiently large $n$, the values of $w_{\tau_n}$
	along any finite chain with total calibration defect at most
	$\tau_n^2$ cannot cross the level $s$ from below. More precisely,
	every transition from $x_k$ to $x_{k+1}$ satisfies
	\[
	w_{\tau_n}(x_k)<s
	\quad\Longrightarrow\quad
	w_{\tau_n}(x_{k+1})<s.
	\]
\end{lemma}

\begin{proof}
Otherwise select such transitions along a subsequence, and denote their
initial points by $y_n$ and their velocities by $v_n$. By
\eqref{eq:small_defect_increment_bound} and uniform convergence,
$w(y_n)\to s$. Passing to a subsequence, $y_n\to y$ with $w(y)=s$.
Since $s\notin w(\mathcal A_L)$, $y\notin\mathcal A_L$, and continuity
gives $Q(y_n,v_n)\geq a(y_n)\geq\delta>0$ for large $n$. Then
\eqref{eq:small_defect_descent} makes the increment at this transition
at most $-\delta\tau_n+C\tau_n^2<0$, a contradiction.
\end{proof}

\begin{theorem}\label{thm:limsup_aubry}
Under Assumption~\ref{ass:finite_static_classes},
\begin{equation}\label{eq:limsup_aubry}
  \limsup_{\tau\to0}\widetilde{\mathcal A}_L^\tau
  \subset\widetilde{\mathcal A}_L.
\end{equation}
Moreover,
\begin{equation}\label{eq:uniform_distance_aubry}
  \lim_{\tau\to0}\sup_{a\in\widetilde{\mathcal A}_L^\tau}
  \dist_X(a,\widetilde{\mathcal A}_L)=0.
\end{equation}
\end{theorem}

\begin{proof}
Let $z_n\in\widetilde{\mathcal A}_L^{\tau_n}$, $\tau_n\to0$, and
$z_n\to z$. Choose static configurations through $z_n$, with velocities
$v_k^n$. Normalize $u_{\tau_n}$ by $\min u_{\tau_n}=0$ and pass to a
subsequence such that $u_{\tau_n}\to u$ uniformly. By
\Cref{lem:compactness_semistatic_configurations,prop:limit_is_semistatic},
the configurations converge as in \eqref{eq:grid_convergence} to a
semistatic trajectory $\gamma$, with
$\widehat\gamma(t):=(\gamma(t),\dot\gamma(t))=\phi_t^L(z)$.

Every static segment is calibrated by $u_{\tau_n}$, as observed in
\Cref{lem:static_small_defect_circuits}. Passing to the limit on bounded
time intervals, using the position and velocity convergence in
\eqref{eq:grid_convergence}, gives
\[
  u(\gamma(b))-u(\gamma(a))
  =\int_a^b\bigl(L(\gamma(t),\dot\gamma(t))+\alpha(H)\bigr)\,dt
  \qquad(a<b).
\]
Consequently, with $w=u^\sharp-u$,
\begin{equation}\label{eq:continuous_descent_identity}
  w(\gamma(b))-w(\gamma(a))
  =-\int_a^b Q(\widehat\gamma(t))\,dt.
\end{equation}
The function $w$ is constant on each static class, so
$w(\mathcal A_L)$ is finite by Assumption~\ref{ass:finite_static_classes}.

Suppose $z\notin\widetilde{\mathcal A}_L$. Then $Q(z)>0$, and
\eqref{eq:continuous_descent_identity} gives $t_*>0$ such that
$w(\gamma(t_*))<w(\gamma(0))$. Choose
\[
  s\in\bigl(w(\gamma(t_*)),w(\gamma(0))\bigr)
    \setminus w(\mathcal A_L).
\]
For $m_n:=\lfloor t_*/\tau_n\rfloor$ and large $n$,
$m_n\geq1$ and
$w_{\tau_n}(x_0^n)>s>w_{\tau_n}(x_{m_n}^n)$.
Complete the static segment from $x_0^n$ to $x_{m_n}^n$ to a closed
chain with total defect at most $\tau_n^2$, using
\Cref{lem:static_small_defect_circuits}. Its return part starts below
$s$ and ends above $s$, so some transition crosses $s$ upwards. This
contradicts \Cref{lem:no_upward_crossing}. Thus
$z\in\widetilde{\mathcal A}_L$, proving \eqref{eq:limsup_aubry}.
Finally, \eqref{eq:uniform_distance_aubry} follows from the common compact
containment in $X_\kappa$ and \Cref{lem:kuratowski_hausdorff}(3).
\end{proof}

\begin{remark}  
	The proof of \Cref{thm:limsup_aubry} uses
	Assumption~\ref{ass:finite_static_classes} only to ensure that
	$(u^\sharp-u)(\mathcal A_L)$ contains no nontrivial interval,
	so that a separating level $s$ can be chosen.
	Thus the same conclusions hold under the weaker condition that
	$(u_1-u_2)(\mathcal A_L)$ contains no nontrivial interval for every
	pair of critical subsolutions $u_1,u_2$.	
	By \cite[Theorem~1.1 and Proposition~4.7]{FathiFigalliRifford2009},
	this condition is automatically satisfied when $d=1,2$ under
	the standing Tonelli assumptions, and when $d=3$ if $H\in C^{3,1}$.
	In these cases, the Aubry upper-limit result therefore holds
	without any assumption on the number of static classes.
\end{remark}

\subsection{Lower Kuratowski limit}\label{subsec:liminf_aubry}

For the componentwise lower-limit statements, we use the following
more specific continuous hypothesis, which implies Assumption~\ref{ass:finite_static_classes}.

\begin{assumption}
\label{ass:finite_aubry_components}
The Aubry set is a finite disjoint union
$\widetilde{\mathcal A}_L=\bigcup_{j=1}^r\Gamma_j$, where each
$\Gamma_j$ is an equilibrium or a nonstationary periodic orbit of
$\phi_t^L$.
\end{assumption}

 Under
Assumption~\ref{ass:finite_aubry_components}, \Cref{thm:limsup_aubry} and
\Cref{lem:kuratowski_hausdorff} reduce the Kuratowski convergence
$\widetilde{\mathcal A}_L^\tau\to\widetilde{\mathcal A}_L$, and at the same time
the convergence in the Hausdorff distance, to the pointwise property
\begin{equation*}
  \dist_X\bigl(z,\widetilde{\mathcal A}_L^\tau\bigr)\longrightarrow0
  \qquad\text{for every }z\in\widetilde{\mathcal A}_L.
\end{equation*}
\Cref{thm:limsup_aubry} shows that discrete static points accumulate on
$\widetilde{\mathcal A}_L$, but not that every component $\Gamma_j$ is
approached. The next lemma uses the shift invariance of static configurations to show
that approaching one point of an equilibrium or periodic component implies
the recovery of the whole component. Indeed, shifting a static configuration
by \(k_\tau\) steps preserves membership in the discrete Aubry set; choosing
\(k_\tau\tau\to t\) produces points converging to \(\phi_t^L(a)\). Since
\(t\) parametrizes the whole component, the entire component is recovered.

\begin{lemma}[Propagation along an orbit]\label{lem:propagation_periodic}
  Let $\Gamma\subset X$ be an equilibrium or a periodic orbit of $\phi^L_t$
  such that
  \begin{equation}\label{eq:approach_periodic}
    \lim_{\tau\to0}\ \inf_{a\in\widetilde{\mathcal A}_L^\tau}
    \dist_X(a,\Gamma)=0.
  \end{equation}
  Then $\Gamma\subset\liminf_{\tau\to0}\widetilde{\mathcal A}_L^\tau$.
\end{lemma}

\begin{proof}
  Fix $z\in\Gamma$ and $\tau_n\downarrow0$. By
  \Cref{lem:kuratowski_hausdorff}(1), it suffices to prove that
  $\dist_X(z,\widetilde{\mathcal A}_L^{\tau_n})\to0$. It is enough to show
  that every subsequence of $(\tau_n)$ has a further subsequence along which
  this holds. We fix a subsequence and do not relabel further extractions.

  By \eqref{eq:approach_periodic}, there are
  $a_n\in\widetilde{\mathcal A}_L^{\tau_n}$ with
  $\dist_X(a_n,\Gamma)\to0$. Since $\Gamma$ is compact, there are
  $g_n\in\Gamma$ with $d_X(a_n,g_n)=\dist_X(a_n,\Gamma)$, and, up to a
  subsequence, $g_n\to a\in\Gamma$. Then $a_n\to a$.

  Choose static configurations $\{(x_k^n,\ell_k^n)\}_{k\in\mathbb Z}$ for
  the step $\tau_n$ with $(x_0^n,v_0^n)=a_n$. By
  \Cref{lem:compactness_semistatic_configurations}, up to a further
  subsequence, \eqref{eq:grid_convergence} holds for an Euler--Lagrange
  trajectory $\gamma$ with $\widehat\gamma(t)=\phi^L_t(a)$. Since $\Gamma$ is
  the orbit of $a$, there is $t_*\in\mathbb R$ with $\phi^L_{t_*}(a)=z$ (if
  $\Gamma$ is an equilibrium, $t_*=0$). Let
  $k_n:=\lfloor t_*/\tau_n\rfloor$ and $b_n:=(x_{k_n}^n,v_{k_n}^n)$. Then
  $b_n\in\widetilde{\mathcal A}_L^{\tau_n}$ by
  \Cref{lem:shift_invariance}, and $b_n\to\widehat\gamma(t_*)=z$ by
  \eqref{eq:grid_convergence}. Hence
  $\dist_X(z,\widetilde{\mathcal A}_L^{\tau_n})\leq d_X(b_n,z)\to0$.
\end{proof}

If the Aubry set is a single equilibrium or periodic orbit,
\Cref{lem:single_static_orbit} and the upper limit of the Ma\~n\'e sets give
the Aubry upper limit directly. Nonemptiness and
\Cref{lem:propagation_periodic} then yield the lower limit.

\begin{theorem} \label{thm:single_component}
  Suppose that $\widetilde{\mathcal A}_L=\Gamma$, where $\Gamma$ is a
  single equilibrium or a single nonstationary periodic orbit. Then
  $\widetilde{\mathcal N}_L=\widetilde{\mathcal A}_L=\Gamma$ and
  \[
    \widetilde{\mathcal A}_L^\tau\longrightarrow\Gamma,
    \qquad
    \widetilde{\mathcal N}_L^\tau\longrightarrow\Gamma
  \]
  in the Kuratowski sense and in the Hausdorff distance.
\end{theorem}

\begin{proof}
  By \Cref{lem:single_static_orbit},
  $\widetilde{\mathcal N}_L=\widetilde{\mathcal A}_L=\Gamma$. Thus
  \Cref{thm:discrete_mane_limsup} and
  $\widetilde{\mathcal A}_L^\tau\subset\widetilde{\mathcal N}_L^\tau$ give
  \[
    \sup_{a\in\widetilde{\mathcal A}_L^\tau}\dist_X(a,\Gamma)
    \leq\sup_{a\in\widetilde{\mathcal N}_L^\tau}\dist_X(a,\Gamma)
    \longrightarrow0.
  \]
  Since
  $\widetilde{\mathcal A}_L^\tau\neq\varnothing$,
  this implies \eqref{eq:approach_periodic}.
  \Cref{lem:propagation_periodic} gives
  $\Gamma\subset\liminf_{\tau\to0}\widetilde{\mathcal A}_L^\tau$, hence
  also $\Gamma\subset\liminf_{\tau\to0}\widetilde{\mathcal N}_L^\tau$.
  Both families lie in $X_\kappa$ by
  \Cref{prop:discrete_set_inclusions}, so
  \Cref{lem:kuratowski_hausdorff} yields both conclusions.
\end{proof}

For several components, \Cref{thm:limsup_aubry} only shows that the
discrete Aubry sets approach the union of the continuous components; it
does not ensure that every component is approached. When all components are
equilibria, the following lemma gives a pointwise criterion for membership
in the discrete Aubry set in terms of the discrete critical value.

\begin{lemma}\label{lem:equilibrium_criterion}
  \begin{enumerate}
    \item If $p=(x,0)\in\widetilde{\mathcal A}_L$ is an equilibrium of
      $\phi^L_t$, then $L(x,0)=-\alpha(H)$.
    \item For every $\tau\in(0,1)$ and $x\in\mathbb T^d$,
      \[
        (x,0)\in\widetilde{\mathcal A}_L^\tau
        \quad\Longleftrightarrow\quad
        \bar L(\tau)=L(x,0).
      \]
  \end{enumerate}
\end{lemma}

\begin{proof}
  (1) Since $p$ is an equilibrium, $\gamma_{x,0}\equiv x$, and this constant
  curve is static because $p\in\widetilde{\mathcal A}_L$. The identity
  \eqref{eq:semistatic} on $[0,1]$ gives $L(x,0)+\alpha(H)=\Phi_L(x,x)$,
  while \eqref{eq:static} on $[0,1]$ gives $2\Phi_L(x,x)=0$. Hence
  $L(x,0)=-\alpha(H)$.

  (2) Assume $(x,0)\in\widetilde{\mathcal A}_L^\tau$ and let
  $\{(x_k,\ell_k)\}_{k\in\mathbb Z}$ be a static configuration with $x_0=x$
  and $v_\tau(x_0,x_1,\ell_0)=0$. Then $\iota(x_1)-\iota(x)+\ell_0=0$. Every
  component of $\iota(x_1)-\iota(x)$ lies in $(-1,1)$, because $\iota$ takes
  values in $[0,1)^d$, while $\ell_0\in\mathbb Z^d$; this forces $\ell_0=0$
  and $\iota(x_1)=\iota(x)$, that is $x_1=x$. Taking $(m,n)=(0,1)$ in
  \eqref{eq:discrete_semistatic} and \eqref{eq:discrete_static} yields
  $c_\tau(x,x,0)=\Phi_\tau(x,x)$ and $2\Phi_\tau(x,x)=0$. Hence
  $\tau L(x,0)-\tau\bar L(\tau)=c_\tau(x,x,0)=0$.

  Conversely, assume $\bar L(\tau)=L(x,0)$ and consider the constant
  configuration $x_k=x$, $\ell_k=0$, whose velocities vanish. Every
  transition has critical cost $c_\tau(x,x,0)=\tau L(x,0)-\tau\bar L(\tau)=0$,
  so every finite sum of consecutive costs vanishes. By
  \Cref{lem:mane_potential} and the definition of $\hat c_\tau$ as a minimum
  over winding numbers,
  \[
    0\leq\Phi_\tau(x,x)\leq\hat c_\tau(x,x)\leq c_\tau(x,x,0)=0.
  \]
  Hence both sides of \eqref{eq:discrete_semistatic} vanish for every
  $m<n$, and so does the left-hand side of \eqref{eq:discrete_static}. The
  configuration is static, and $(x,0)\in\widetilde{\mathcal A}_L^\tau$.
\end{proof}

The mechanical model discussed in \Cref{sec:mechanical_model}
illustrates the criterion of \Cref{lem:equilibrium_criterion}:
the identity $\bar L(\tau)=-\alpha(H)$ can be verified directly,
ensuring that every continuous Aubry equilibrium also belongs
to the discrete Aubry set. In this case, both sets can be computed
explicitly and coincide for every $\tau\in(0,1)$.
For general Lagrangians whose Aubry set consists of equilibria, we use the
following nondegeneracy condition.

\begin{assumption} \label{ass:quadratic_critical_gap}
  Let $\mathcal W\subset X$ be a compact set. There exist
  $u_*\in C^2(\mathbb T^d)$ and $c>0$ such that the function
  \[
    \ell_*(x,v):=L(x,v)+\alpha(H)-Du_*(x)\cdot v
  \]
  satisfies
  \begin{equation}\label{eq:sd_quadratic_gap}
    \ell_*(x,v)\geq c\,\dist_X\bigl((x,v),\widetilde{\mathcal A}_L\bigr)^2
    \qquad\text{for every }(x,v)\in\mathcal W.
  \end{equation}
  The set $\mathcal W$ specifies the region on which the gap is required.
\end{assumption}

Natural candidates for $u_*$ are $C^2$ critical subsolutions, for which
$\ell_*\geq0$ on $X$ by the Fenchel inequality and $\ell_*=0$ on
$\widetilde{\mathcal A}_L$ \cite{Fathi2008}. The quadratic lower bound
\eqref{eq:sd_quadratic_gap} is a nondegeneracy condition which does not follow
from these properties. A mechanical example with a global quadratic gap is
given in \Cref{ex:quadratic_gap}.

For finitely many equilibrium Aubry components, we prove full convergence
by establishing $\bar L(\tau)=-\alpha(H)$ and
$\widetilde{\mathcal A}_L\subset\widetilde{\mathcal A}_L^\tau$
for small $\tau$, then applying the upper-limit theorem.
The quadratic gap provides a positive lower bound on the corrected
critical action, allowing us to absorb the discretization error and
prove the critical-value identity. This requires controlling the
average of $Du_*(x)\cdot v$: discrete holonomy cancels finite differences
rather than differentials, leaving a Taylor remainder of order
$\tau|v|^2$ after division by $\tau$.
The following estimate quantifies this error.

\begin{lemma} \label{lem:holonomy_estimate}
  Let $\mu$ be a $\tau$-holonomic probability measure with compact support.
  Then, for every $u\in C^2(\mathbb T^d)$,
  \begin{equation}\label{eq:sd_holonomy_estimate}
    \left|\int Du(x)\cdot v\,d\mu(x,v)\right|
    \leq\frac{\tau}{2}\,\|D^2u\|_\infty\int|v|^2\,d\mu(x,v).
  \end{equation}
\end{lemma}

\begin{proof}
  Let $\widetilde u$ be the periodic lift of $u$. For $(x,v)\in X$ and a lift
  $\widetilde x$ of $x$, Taylor's formula with integral remainder gives
  \[
    u\bigl(T_\tau(x,v)\bigr)-u(x)
    =\widetilde u(\widetilde x+\tau v)-\widetilde u(\widetilde x)
    =\tau Du(x)\cdot v+R_\tau(x,v),
  \]
  where
  \[
    R_\tau(x,v)=\tau^2\int_0^1(1-s)\,D^2\widetilde u(\widetilde x+s\tau v)[v,v]\,ds,
    \qquad
    |R_\tau(x,v)|\leq\frac{\tau^2}{2}\|D^2u\|_\infty|v|^2.
  \]
  All these functions are continuous, hence bounded on the compact support
  of $\mu$. Integrating against $\mu$ and using \eqref{eq:discrete_holonomy}
  with $\varphi=u$, we obtain
  \[
    0=\tau\int Du(x)\cdot v\,d\mu+\int R_\tau\,d\mu.
  \]
  Therefore
  $\tau\bigl|\int Du\cdot v\,d\mu\bigr|\leq\int|R_\tau|\,d\mu
  \leq\frac{\tau^2}2\|D^2u\|_\infty\int|v|^2\,d\mu$, which is
  \eqref{eq:sd_holonomy_estimate}.
\end{proof}

\begin{theorem}\label{thm:convergence_equilibria}
	Assume that $\widetilde{\mathcal A}_L$ consists of equilibria and that
	Assumption~\ref{ass:quadratic_critical_gap} holds on a compact set
	$\mathcal W$. Suppose that either
	\begin{enumerate}
		\item[(a)] $\mathcal W=X_\kappa$; or
		\item[(b)] $\widetilde{\mathcal A}_L\subset
		\operatorname{int}_X\mathcal W$ and
		Assumption~\ref{ass:finite_static_classes} holds.
	\end{enumerate}
	Then there exists $\tau_{\rm loc}\in(0,1]$ such that
	$\widetilde{\mathcal A}_L^\tau\subset\mathcal W$ for
	$0<\tau<\tau_{\rm loc}$, with $\tau_{\rm loc}=1$ in case~(a).
	Set
	\[
	\tau_*:=\min\left\{\tau_{\rm loc},
	\frac{2c}{\|D^2u_*\|_\infty}\right\},
	\]
	where the quotient is interpreted as $+\infty$ if $D^2u_*\equiv0$.
	Then:
	\begin{enumerate}
		\item For every $0<\tau<\tau_*$,
		\[
		\bar L(\tau)=-\alpha(H),
		\qquad
		\widetilde{\mathcal A}_L
		\subset\widetilde{\mathcal A}_L^\tau.
		\]
		\item If Assumption~\ref{ass:finite_static_classes} holds, then
		$\widetilde{\mathcal A}_L^\tau\to\widetilde{\mathcal A}_L$
		as $\tau\to0$, both in the Kuratowski sense and in the
		Hausdorff distance.
	\end{enumerate}
\end{theorem}
\begin{proof}
  In case~(a), \Cref{prop:discrete_set_inclusions} gives
  $\widetilde{\mathcal A}_L^\tau\subset X_\kappa=\mathcal W$ for every
  $\tau\in(0,1)$, so we take $\tau_{\rm loc}=1$.
  In case~(b), Assumption~\ref{ass:finite_static_classes} allows us to
  apply \Cref{thm:limsup_aubry}. Since the compact set
  $\widetilde{\mathcal A}_L$ lies in $\operatorname{int}_X\mathcal W$,
  the uniform distance estimate \eqref{eq:uniform_distance_aubry} yields
  $\tau_{\rm loc}\in(0,1]$ such that
  $\widetilde{\mathcal A}_L^\tau\subset\mathcal W$ for
  $0<\tau<\tau_{\rm loc}$. Fix $\tau_{\rm loc}$ accordingly.

  (1) Choose $(x,0)\in\widetilde{\mathcal A}_L$, which is nonempty by
  \Cref{prop:continuous_sets}. By \eqref{eq:ergodic_constant_bounds}
  and \Cref{lem:equilibrium_criterion}(1),
  $\bar L(\tau)\leq L(x,0)=-\alpha(H)$ for every $\tau\in(0,1)$.

  For the reverse inequality, let $0<\tau<\tau_*$ and let $\mu_\tau$ be a
  discrete Mather measure. By \Cref{lem:minimum_action} and
  \Cref{prop:discrete_set_inclusions},
  $\int L\,d\mu_\tau=\bar L(\tau)$ and
  $\operatorname{spt}(\mu_\tau)\subset\widetilde{\mathcal M}_L^\tau
  \subset\widetilde{\mathcal A}_L^\tau\subset\mathcal W$; in particular
  $\mu_\tau$ has compact support. Since
  $\widetilde{\mathcal A}_L\subset\mathbb T^d\times\{0\}$, for every
  $(x,v)\in X$
  \[
    \dist_X\bigl((x,v),\widetilde{\mathcal A}_L\bigr)^2
    =\min_{(y,0)\in\widetilde{\mathcal A}_L}\bigl(d_{\mathbb T^d}(x,y)^2+|v|^2\bigr)
    \geq|v|^2,
  \]
  so \eqref{eq:sd_quadratic_gap} gives $\ell_*(x,v)\geq c|v|^2$ on
  $\operatorname{spt}(\mu_\tau)$. Using the definition of $\ell_*$ and
  \Cref{lem:holonomy_estimate} with $u=u_*$,
  \[
    \bar L(\tau)+\alpha(H)
    =\int\bigl(L+\alpha(H)\bigr)\,d\mu_\tau
    =\int\ell_*\,d\mu_\tau+\int Du_*(x)\cdot v\,d\mu_\tau
    \geq\Bigl(c-\frac{\tau}{2}\|D^2u_*\|_\infty\Bigr)\int|v|^2\,d\mu_\tau.
  \]
  By the definition of $\tau_*$, the coefficient on the right is positive
  for $0<\tau<\tau_*$, also when $D^2u_*\equiv0$. Thus
  $\bar L(\tau)\geq-\alpha(H)$, and the reverse inequality proved above
  gives $\bar L(\tau)=-\alpha(H)$.

  Now let $(x,0)\in\widetilde{\mathcal A}_L$. By
  \Cref{lem:equilibrium_criterion}(1), $L(x,0)=-\alpha(H)=\bar L(\tau)$, and
  \Cref{lem:equilibrium_criterion}(2) gives
  $(x,0)\in\widetilde{\mathcal A}_L^\tau$.

  (2) Assume now that Assumption~\ref{ass:finite_static_classes} holds.
  Part~(1) gives
  $\widetilde{\mathcal A}_L\subset
  \liminf_{\tau\to0}\widetilde{\mathcal A}_L^\tau$, whereas
  \Cref{thm:limsup_aubry} gives the reverse upper inclusion.
  Since $\widetilde{\mathcal A}_L^\tau\subset\mathcal W$ for
  $0<\tau<\tau_{\rm loc}$, \Cref{lem:kuratowski_hausdorff}(4) yields
  both Kuratowski and Hausdorff convergence.
\end{proof}

\begin{remark} 
	\label{rem:periodic_open}
	Assume that
	Assumption~\ref{ass:finite_aubry_components} holds, with $r\geq2$ and at least one
	periodic component. By
	\Cref{lem:propagation_periodic},
	\Cref{lem:kuratowski_hausdorff}, and
	\Cref{thm:limsup_aubry}, full Kuratowski convergence is equivalent to the
	persistence of every component:
	\begin{equation}\label{eq:sd_periodic_persistence}
		\lim_{\tau\to0}\ \inf_{a\in\widetilde{\mathcal A}_L^\tau}
		\dist_X(a,\Gamma_j)=0,
		\qquad j=1,\dots,r.
	\end{equation}
	Indeed, the upper-limit theorem only implies that the discrete Aubry sets
	approach the union $\bigcup_j\Gamma_j$; it does not prevent them from
	approaching only a proper subfamily of the components along a sequence
	$\tau_n\downarrow0$. We do not know whether \eqref{eq:sd_periodic_persistence}
	holds in this generality.
	
	For equilibrium components, persistence follows from the identity
	$\bar L(\tau)=-\alpha(H)$. This argument is not available for periodic
	components. To see the obstruction, let $\gamma:\mathbb R\to\mathbb T^d$
	be a static trajectory such that
	$\widehat\gamma(t):=(\gamma(t),\dot\gamma(t))$ is periodic
	of period $T$ and takes values in $\widetilde{\mathcal A}_L$.
	Let $\widetilde\gamma:\mathbb R\to\mathbb R^d$ be a lift of $\gamma$. For $\tau=T/N$, define
	\[
	x_k:=\gamma(k\tau),\qquad
	v_k:=\frac{\widetilde\gamma((k+1)\tau)
		-\widetilde\gamma(k\tau)}{\tau},
	\qquad
	\mu_N:=\frac1N\sum_{k=0}^{N-1}\delta_{(x_k,v_k)}.
	\]
	Then $\mu_N$ is $\tau$-holonomic. A Taylor expansion, combined with the
	Euler--Lagrange equation and
	\[
	\int_0^T\bigl(L(\gamma(t),\dot\gamma(t))+\alpha(H)\bigr)\,dt=0,
	\]
	gives
	\[
	\int L\,d\mu_N
	=
	-\alpha(H)
	-\frac{\tau}{2T}\int_0^T
	D_xL(\gamma(t),\dot\gamma(t))\cdot\dot\gamma(t)\,dt
	+o(\tau).
	\]
	Since $\bar L(\tau)\leq\int L\,d\mu_N$, if the integral in the correction
	term is positive, then
	$\bar L(\tau)<-\alpha(H)$ for all sufficiently small $\tau=T/N$.
	Thus the exact-critical-value mechanism used for equilibria cannot in general
	force the persistence of periodic components.
	For mechanical Lagrangians $L=\frac12|v|^2-U(x)$ this first-order correction vanishes,
	since $D_xL(\gamma,\dot\gamma)\cdot\dot\gamma=-\frac{d}{dt}U(\gamma)$ integrates to
	zero over a period.
	
	A natural possible approach near hyperbolic periodic components is to use
	shadowing \cite{Palmer2000}. However, shadowing produces discrete
	Euler--Lagrange orbits close to the continuous periodic orbit, whereas
	\eqref{eq:sd_periodic_persistence} requires discrete static configurations.
	The latter satisfy a global minimizing property that is not provided by
	shadowing alone.
\end{remark}


\section{Convergence of Mather measures and of the Mather set}\label{sec:conv_mather}

We first study weak limits of exact discrete Mather measures. We then record
the consequences for the lifted Mather sets, using the inclusions between the
discrete Ma\~n\'e, Aubry and Mather sets proved above. For $\tau\in(0,1)$, let
\begin{equation}
  \mathscr M_\tau
  :=
  \left\{
    \nu\in\mathcal H_\tau:
    \int_X L\,d\nu=\bar L(\tau)
  \right\}
\end{equation}
denote the set of discrete Mather measures. We also set
\begin{equation}
  \mathscr M
  :=
  \left\{
    \mu\in\mathcal P_1(X):
    \mu\text{ is closed and }
    \int_X L\,d\mu=-\alpha(H)
  \right\}.
\end{equation}
Thus $\mathscr M_\tau$ and $\mathscr M$ are sets of probability measures,
whereas $\widetilde{\mathcal M}_L^\tau$ and $\widetilde{\mathcal M}_L$ are
subsets of $X$. We use the weak topology on $\mathcal P(X)$, that is, the
topology of convergence against bounded continuous functions. For a metric
$d_*$ inducing this topology and a subset $\mathcal F\subset\mathcal P(X)$,
we write
\[
  \dist_{d_*}(\nu,\mathcal F):=\inf_{\mu\in\mathcal F}d_*(\nu,\mu).
\]

\begin{proposition}
\label{prop:mather_measure_limsup}
Let $\tau_n\to0$ and $\mu_n\in\mathscr M_{\tau_n}$. Then, up to a
subsequence, $\mu_n\rightharpoonup\mu$ for some $\mu\in\mathscr M$.
Consequently,
\begin{equation}
  \limsup_{\tau\to0}\mathscr M_\tau\subset\mathscr M
\end{equation}
with respect to the weak topology. Moreover, if $d_*$ is any metric inducing
the weak topology on a common compact set containing all the measures under
consideration, then
\begin{equation}\label{eq:mather_measure_uniform_upper}
  \sup_{\nu\in\mathscr M_\tau}
  \dist_{d_*}(\nu,\mathscr M)\longrightarrow0.
\end{equation}
\end{proposition}

\begin{proof}
By \Cref{lem:minimum_action}, every $\mu_n$ is supported in
$X_\kappa$. Thus the
measures $\mu_n$ belong to the compact space $\mathcal P(X_\kappa)$, and a
subsequence converges weakly to a probability measure $\mu$ supported in
$X_\kappa$. Fix $\varphi\in C^1(\mathbb T^d)$ and define
\[
  g_n(x,v):=
  \frac{\varphi(T_{\tau_n}(x,v))-\varphi(x)}{\tau_n}.
\]
On $X_\kappa$, $g_n$ converges uniformly to
$g(x,v):=D\varphi(x)\cdot v$: indeed
$g_n(x,v)=\int_0^1D\widetilde\varphi(\widetilde x+s\tau_nv)\cdot v\,ds$, so
$|g_n-g|\le\kappa\,\omega_{D\varphi}(\kappa\tau_n)$, where $\omega_{D\varphi}$ is a
modulus of continuity of $D\varphi$. Since $\mu_n$ is $\tau_n$-holonomic,
$\int g_n\,d\mu_n=0$, and therefore
\[
  \int_XD\varphi(x)\cdot v\,d\mu(x,v)=0.
\]
Hence $\mu$ is closed. Moreover, $L$ is bounded and continuous on $X_\kappa$, so
\[
  \int_XL\,d\mu
  =\lim_{n\to\infty}\int_XL\,d\mu_n
  =\lim_{n\to\infty}\bar L(\tau_n)
  =-\alpha(H).
\]
Thus $\mu\in\mathscr M$. This proves the sequential upper-limit statement.
The uniform estimate \eqref{eq:mather_measure_uniform_upper} follows by
compactness: otherwise one could find $\tau_n\to0$ and
$\mu_n\in\mathscr M_{\tau_n}$ staying at a fixed positive distance from
$\mathscr M$, contradicting the result just proved.
\end{proof}

\begin{remark}

The preceding proposition gives only an upper-limit result. In general, we
do not know whether every continuous Mather measure can be approximated by
exact discrete Mather measures. The discretization may select a proper
subset of the continuous minimizing measures, so that
$\mathscr M\subset\liminf_{\tau\to0}\mathscr M_\tau$ cannot be asserted
without additional assumptions.
\end{remark}

The set-valued consequences are direct from the inclusions already proved.

\begin{corollary}
\label{cor:discrete_mather_upper_limits}
The discrete Mather sets satisfy
\begin{equation}
  \limsup_{\tau\to0}\widetilde{\mathcal M}_L^\tau
  \subset\widetilde{\mathcal N}_L,
  \qquad
  \lim_{\tau\to0}
  \sup_{z\in\widetilde{\mathcal M}_L^\tau}
  \dist_X\bigl(z,\widetilde{\mathcal N}_L\bigr)=0.
\end{equation}
Under Assumption~\ref{ass:finite_aubry_components}, one has
$\widetilde{\mathcal M}_L=\widetilde{\mathcal A}_L$ and
\begin{equation}
  \limsup_{\tau\to0}\widetilde{\mathcal M}_L^\tau
  \subset\widetilde{\mathcal M}_L,
  \qquad
  \lim_{\tau\to0}
  \sup_{z\in\widetilde{\mathcal M}_L^\tau}
  \dist_X\bigl(z,\widetilde{\mathcal M}_L\bigr)=0.
\end{equation}
\end{corollary}

\begin{proof}
By \Cref{prop:discrete_set_inclusions},
\[
  \widetilde{\mathcal M}_L^\tau
  \subset\widetilde{\mathcal A}_L^\tau
  \subset\widetilde{\mathcal N}_L^\tau.
\]
The first assertion follows from \Cref{thm:discrete_mane_limsup}.
Under Assumption~\ref{ass:finite_aubry_components}, every equilibrium
component supports a minimizing Dirac mass, and every periodic component
supports its normalized invariant orbit measure. Indeed, staticity makes
the critical action over a period zero, and the same identity holds over
any time interval for a constant static trajectory. Each component is
therefore contained in $\widetilde{\mathcal M}_L$, so
$\widetilde{\mathcal A}_L=\widetilde{\mathcal M}_L$ by
\eqref{eq:continuous_inclusions}. The second assertion now follows from
\Cref{thm:limsup_aubry}.
\end{proof}

\begin{corollary} 
\label{cor:full_convergence_unique_mather}
Assume that $\mathscr M=\{\mu\}$ and
$\operatorname{spt}(\mu)=\widetilde{\mathcal N}_L$. Then
$\widetilde{\mathcal M}_L=\widetilde{\mathcal A}_L
=\widetilde{\mathcal N}_L$, and
\[
  d_H\bigl(\widetilde{\mathcal M}_L^\tau,\widetilde{\mathcal M}_L\bigr)
  +d_H\bigl(\widetilde{\mathcal A}_L^\tau,\widetilde{\mathcal A}_L\bigr)
  +d_H\bigl(\widetilde{\mathcal N}_L^\tau,\widetilde{\mathcal N}_L\bigr)
  \longrightarrow0.
\]
Moreover, every choice $\mu_\tau\in\mathscr M_\tau$ converges weakly to
$\mu$ as $\tau\to0$.
\end{corollary}

\begin{proof}
The continuous identities follow from \eqref{eq:continuous_inclusions}.
Choose $\mu_\tau\in\mathscr M_\tau$. By
\Cref{prop:mather_measure_limsup} and uniqueness,
$\mu_\tau\rightharpoonup\mu$. For every
$z\in\operatorname{spt}(\mu)$ and $\varepsilon>0$, the open ball
$B_X(z,\varepsilon)$ has positive $\mu$-mass. The Portmanteau theorem
therefore gives $\mu_\tau(B_X(z,\varepsilon))>0$ for all sufficiently
small $\tau$. Since
$\operatorname{spt}(\mu_\tau)\subset\widetilde{\mathcal M}_L^\tau$, this
implies $\dist_X(z,\widetilde{\mathcal M}_L^\tau)\to0$. Hence
\[
  \widetilde{\mathcal N}_L
  \subset\liminf_{\tau\to0}\widetilde{\mathcal M}_L^\tau
  \subset\limsup_{\tau\to0}\widetilde{\mathcal N}_L^\tau
  \subset\widetilde{\mathcal N}_L,
\]
where the last inclusion is \Cref{thm:discrete_mane_limsup}.
The discrete inclusions in \Cref{prop:discrete_set_inclusions} yield
Kuratowski convergence of all three families. Their common compact
containment in $X_\kappa$ and \Cref{lem:kuratowski_hausdorff} give
Hausdorff convergence.
\end{proof}

\begin{remark} 
Zhang \cite{Zhang2017} proves that, for a fixed autonomous Tonelli
Lagrangian $L$, there is a residual subset $\mathcal O\subset
C^\infty(\mathbb T^d)$ of potentials such that, for every $f\in\mathcal O$,
the Mather, Aubry and Ma\~n\'e sets of $L+f-\eta_c$ coincide and the
minimizing measure is unique for a residual set of classes
$c\in H^1(\mathbb T^d;\mathbb R)$; here $\eta_c$ is a smooth closed
one-form representing $c$. Thus
\Cref{cor:full_convergence_unique_mather} applies to these Lagrangians.
Since a residual set of cohomology classes need not contain the zero class, this genericity result alone does not establish the hypotheses
of the preceding corollary for the unmodified Lagrangian $L+f$.
\end{remark}

For equilibrium Aubry components, the quadratic gap yields exact
preservation of the Mather sets and measures for small time steps.

\begin{corollary}
\label{cor:mather_equilibrium_convergence}
Assume the hypotheses of \Cref{thm:convergence_equilibria}(1). Then, for
every $0<\tau<\tau_*$,
\begin{equation}\label{eq:mather_equilibrium_sandwich}
  \widetilde{\mathcal A}_L
  =
  \widetilde{\mathcal M}_L
  =
  \widetilde{\mathcal M}_L^\tau
  \subset
  \widetilde{\mathcal A}_L^\tau,
\end{equation}
and
\begin{equation}\label{eq:mather_equilibrium_convergence}
  \mathscr M_\tau=\mathscr M.
\end{equation}
\end{corollary}

\begin{proof}
Fix $0<\tau<\tau_*$ and $p=(x,0)\in\widetilde{\mathcal A}_L$. By
\Cref{thm:convergence_equilibria}(1) and
\Cref{lem:equilibrium_criterion}(1),
\[
  \bar L(\tau)=-\alpha(H)=L(x,0).
\]
Since $T_\tau(x,0)=x$, the Dirac mass $\delta_p$ is $\tau$-holonomic and
has action $\bar L(\tau)$. It is therefore a discrete Mather measure. It is
also a continuous Mather measure, so
\[
  \widetilde{\mathcal A}_L
  \subset\widetilde{\mathcal M}_L\cap\widetilde{\mathcal M}_L^\tau.
\]
Together with \eqref{eq:continuous_inclusions}, this proves
$\widetilde{\mathcal M}_L=\widetilde{\mathcal A}_L
\subset\widetilde{\mathcal M}_L^\tau$.

If $\mu\in\mathscr M$, then its support is contained in
$\widetilde{\mathcal M}_L=\widetilde{\mathcal A}_L
\subset\mathbb T^d\times\{0\}$. Hence $\mu$ is $\tau$-holonomic and, since
$\bar L(\tau)=-\alpha(H)$, it has discrete minimizing action. Thus
$\mu\in\mathscr M_\tau$.

Conversely, let $\mu_\tau\in\mathscr M_\tau$. Its support is contained in
$\mathcal W$ by the compact containment established in the proof of
\Cref{thm:convergence_equilibria}. Since
$|v|^2\leq\dist_X((x,v),\widetilde{\mathcal A}_L)^2$, the quadratic gap
and \Cref{lem:holonomy_estimate} give
\[
  0=\bar L(\tau)+\alpha(H)
  \geq\left(c-\frac{\tau}{2}\|D^2u_*\|_\infty\right)
  \int_X\dist_X(z,\widetilde{\mathcal A}_L)^2\,d\mu_\tau(z).
\]
The coefficient is positive for $\tau<\tau_*$, hence
$\operatorname{spt}(\mu_\tau)\subset\widetilde{\mathcal A}_L
\subset\mathbb T^d\times\{0\}$. Therefore $\mu_\tau$ is closed and has
continuous minimizing action $-\alpha(H)$, so $\mu_\tau\in\mathscr M$.
This proves \eqref{eq:mather_equilibrium_convergence}; taking the closures
of the unions of supports proves \eqref{eq:mather_equilibrium_sandwich}.
\end{proof}

For periodic components, the discretization of an invariant measure supported
on a continuous periodic static orbit is $\tau$-holonomic but need not be a
discrete Mather measure. This is the measure-theoretic counterpart of the
obstruction described in \Cref{rem:periodic_open}.
\section{A mechanical model and applications}
\label{sec:mechanical_model}

We collect here the mechanical examples used to illustrate the preceding
results. Consider first
\[
  L(x,v)=\frac12|v|^2-U(x),
  \qquad U\in C^2(\mathbb T^d),
\]
and set
\[
  \overline U:=\max_{\mathbb T^d}U,
  \qquad
  Z:=\{x\in\mathbb T^d:U(x)=\overline U\}.
\]

\begin{example}\label{ex:mechanical_example}
  For the mechanical Lagrangian, $\alpha(H)=\overline U$ and
  \[
    \widetilde{\mathcal M}_L
    =\widetilde{\mathcal A}_L
    =\widetilde{\mathcal A}_L^\tau
    =\widetilde{\mathcal M}_L^\tau
    =Z\times\{0\}
    \qquad\text{for every }\tau\in(0,1).
  \]
  Indeed, $L+\overline U=\frac12|v|^2+\overline U-U(x)$ is nonnegative
  and vanishes precisely on $Z\times\{0\}$. This gives the continuous
  identities $\alpha(H)=\overline U$ and
  $\widetilde{\mathcal A}_L=Z\times\{0\}$ \cite{Fathi2008}. Every closed
  continuous or $\tau$-holonomic measure has action at least
  $-\overline U$, whereas
  $\bar L(\tau)\leq\min_xL(x,0)=-\overline U$ by
  \eqref{eq:ergodic_constant_bounds}. Hence
  $\bar L(\tau)=-\overline U$ by \Cref{lem:minimum_action}. Moreover,
  \[
    c_\tau(x,y,\ell)
    =\tau\,(L+\overline U)\bigl(x,v_\tau(x,y,\ell)\bigr)\geq0,
  \]
  so $\Phi_\tau\geq0$. For a static configuration and $m<n$,
  \[
    0\leq\sum_{k=m}^{n-1}c_\tau(x_k,x_{k+1},\ell_k)
    =\Phi_\tau(x_m,x_n)=-\Phi_\tau(x_n,x_m)\leq0.
  \]
  Thus every transition of a discrete static configuration belongs to
  $Z\times\{0\}$, and the same zero-action argument applies to minimizing
  measures. Conversely, for $x\in Z$, the constant transition $(x,0)$ is
  discrete static and minimizing by \Cref{lem:equilibrium_criterion}(2).
\end{example}

The preceding example identifies the discrete Aubry and Mather sets exactly,
without nondegeneracy assumptions. We next consider a perturbation by
an exact differential, which leaves the continuous dynamics unchanged but is
not exact for the time-discrete action.

\begin{example}
\label{ex:quadratic_gap}
Let $U\in C^2(\mathbb T^d)$ attain its maximum $\overline U$ exactly on the
finite set $Z=\{x_1,\dots,x_r\}$, and assume that
$D^2U(x_j)$ is negative definite for every $j$. Let
$g\in C^3(\mathbb T^d)$ and consider
\[
  L(x,v):=\frac12|v|^2-U(x)+Dg(x)\cdot v.
\]
The term $Dg(x)\cdot v$ is an exact differential. Hence $L$ and
$L_0(x,v):=\frac12|v|^2-U(x)$ have the same Euler--Lagrange flow, critical
value and static trajectories, while
\[
  \Phi_L(x,y)=\Phi_{L_0}(x,y)+g(y)-g(x)
  \cite{Fathi2008}.
\]
By \Cref{ex:mechanical_example},
\[
  \alpha(H)=\overline U,
  \qquad
  \widetilde{\mathcal A}_L=Z\times\{0\}.
\]

Choose $u_*:=g$. Then
\[
  \ell_*(x,v)
  =L(x,v)+\alpha(H)-Dg(x)\cdot v
  =\frac12|v|^2+\overline U-U(x).
\]
The nondegeneracy of the maxima and compactness give a constant $c_0>0$
such that
\[
  \overline U-U(x)
  \geq c_0\,\dist_{\mathbb T^d}(x,Z)^2.
\]
Since
\[
  \dist_X\bigl((x,v),Z\times\{0\}\bigr)^2
  =\dist_{\mathbb T^d}(x,Z)^2+|v|^2,
\]
we obtain
\[
  \ell_*(x,v)
  \geq c\,\dist_X\bigl((x,v),\widetilde{\mathcal A}_L\bigr)^2,
  \qquad
  c:=\min\{c_0,\tfrac12\}.
\]
Thus the quadratic critical gap holds globally.
Moreover, the linearization of the Euler--Lagrange equation at $(x_j,0)$ is
$\ddot\xi=-D^2U(x_j)\xi$. Thus the equilibria are hyperbolic; the finite
set $Z\times\{0\}$ satisfies
Assumptions~\ref{ass:finite_aubry_components} and~\ref{ass:finite_static_classes}.

The choice $u_*=g$ is relevant. If $Dg(x_j)\neq0$, then for
$0<s<2$,
\[
  L(x_j,-sDg(x_j))+\alpha(H)
  =\left(\frac{s^2}{2}-s\right)|Dg(x_j)|^2<0.
\]
Consequently, $u_*\equiv0$ cannot satisfy the quadratic lower bound near
$(x_j,0)$. For finite $Z$, nondegeneracy is also necessary for a quadratic
gap. Indeed, if $D^2U(x_j)$ is degenerate, choose a unit vector
$e\in\ker D^2U(x_j)$. In local coordinates, Taylor's formula gives
\[
  \overline U-U(x_j+te)=o(t^2),
  \qquad \dist_{\mathbb T^d}(x_j+te,Z)^2=t^2
\]
for small $t$. Thus the required estimate fails at $(x_j+te,0)$,
independently of the choice of $u_*$.
\end{example}

The mechanical structure gives a stronger conclusion than the general
quadratic-gap theorem: the Aubry and Mather sets are preserved exactly for
small time steps, even without nondegenerate maxima.

\begin{proposition}\label{prop:mechanical_gap}

Let $U\in C^2(\mathbb T^d)$, $g\in C^3(\mathbb T^d)$ and
$L(x,v)=\frac12|v|^2-U(x)+Dg(x)\cdot v$. Set
\[
  K_g:=\|D^2g\|_\infty,
  \qquad
  \tau_g:=
  \begin{cases}
    \min\{1,K_g^{-1}\}, & \text{if }K_g>0,\\
    1, & \text{if }K_g=0.
  \end{cases}
\]
Then, for every $0<\tau<\tau_g$, $\bar L(\tau)=-\overline U$ and
\[
  \widetilde{\mathcal A}_L^\tau
  =\widetilde{\mathcal M}_L^\tau
  =\widetilde{\mathcal A}_L
  =\widetilde{\mathcal M}_L
  =Z\times\{0\}.
\]
Moreover, $\mathscr M_\tau=\mathscr M$ and  both families consist of all Borel
probability measures supported in $Z\times\{0\}$. 
\end{proposition}

\begin{proof}
The exact differential preserves the continuous minimizing measures and
static trajectories, so \Cref{ex:mechanical_example} gives the
continuous identities. Fix $0<\tau<\tau_g$ and define the corrected cost
\[
  q_\tau(x,v):=\tau\bigl(L(x,v)+\overline U\bigr)
    -g(T_\tau(x,v))+g(x).
\]
Taylor's formula on a compatible lift yields
\begin{equation}\label{eq:mechanical_corrected_cost}
  q_\tau(x,v)\geq\tau\left[
    \frac{1-\tau K_g}{2}|v|^2+\overline U-U(x)\right]\geq0
  \qquad\text{on }X.
\end{equation}
The coefficient of $|v|^2$ is positive, so $q_\tau=0$ precisely on
$Z\times\{0\}$. Integrating against a discrete Mather measure and using
holonomy gives
$\tau(\bar L(\tau)+\overline U)=\int q_\tau\,d\mu_\tau\geq0$.
The reverse inequality follows by testing with $\delta_{(x,0)}$, $x\in Z$.
Hence $\bar L(\tau)=-\overline U$, and every discrete Mather measure is
supported in $Z\times\{0\}$. Conversely, every probability measure
supported there is $\tau$-holonomic with minimizing action. This proves
the identities for Mather measures and sets.

Since subtracting the increment of $g$ changes any chain cost only by its
endpoint values, \eqref{eq:mechanical_corrected_cost} implies
$\Phi_\tau(x,y)\geq g(y)-g(x)$. For a static configuration, writing
$v_k=v_\tau(x_k,x_{k+1},\ell_k)$, we obtain
\[
  0\leq\sum_{k=m}^{n-1}q_\tau(x_k,v_k)
  =\Phi_\tau(x_m,x_n)-g(x_n)+g(x_m)
  =-\Phi_\tau(x_n,x_m)-g(x_n)+g(x_m)\leq0.
\]
Taking $n=m+1$ shows that every transition lies in $Z\times\{0\}$.
The reverse Aubry inclusion follows from the constant static
configurations at points of $Z$, by \Cref{lem:equilibrium_criterion}(2).
\end{proof}

If $Z=\{x_*\}$, then
$\widetilde{\mathcal N}_L=\widetilde{\mathcal A}_L=\{(x_*,0)\}$
by \Cref{lem:single_static_orbit}. Thus
\Cref{thm:single_component} also gives Kuratowski and Hausdorff
convergence of the discrete Ma\~n\'e sets.

\section*{Acknowledgments}

The authors were partially supported by the Istituto Nazionale di Alta
Matematica (INdAM), through GNAMPA.


\appendix

\section{Convergence of sets}\label{sec:set_convergence}

For a family
$\{S^\tau\}_{\tau\in(0,1)}$ of subsets of $X$,
\begin{align*}
  \limsup_{\tau\to0}S^\tau
  &:=\bigl\{z\in X:\ \text{there exist }\tau_n\downarrow0
      \text{ and }z_n\in S^{\tau_n}\text{ with }z_n\to z\bigr\},\\
  \liminf_{\tau\to0}S^\tau
  &:=\bigl\{z\in X:\ \text{for every }\tau_n\downarrow0
      \text{ there exist }z_n\in S^{\tau_n}\text{ with }z_n\to z\bigr\},
\end{align*}
and $S^\tau\to S$ in the Kuratowski sense if both limits coincide with $S$.
Clearly $\liminf_{\tau\to0}S^\tau\subset\limsup_{\tau\to0}S^\tau$. For
nonempty compact sets $S,S'\subset X$ we denote by
\[
  d_H(S,S'):=\max\Bigl\{\sup_{a\in S}\dist_X(a,S'),\ \sup_{b\in S'}\dist_X(b,S)\Bigr\}
\]
their Hausdorff distance. The following elementary facts will be used
repeatedly.

\begin{lemma}\label{lem:kuratowski_hausdorff}
  Let $\{S^\tau\}_{\tau\in(0,1)}$ be nonempty compact subsets of $X$ and let
  $S\subset X$ be nonempty and compact.
  \begin{enumerate}
    \item A point $z\in X$ belongs to $\liminf_{\tau\to0}S^\tau$ if and only
      if $\dist_X(z,S^\tau)\to0$ as $\tau\to0$.
    \item If $S\subset\liminf_{\tau\to0}S^\tau$, then
      $\sup_{z\in S}\dist_X(z,S^\tau)\to0$ as $\tau\to0$.
    \item If $\limsup_{\tau\to0}S^\tau\subset S$ and there exist a compact set
      $\mathcal W\subset X$ and $\tau_1\in(0,1)$ such that
      $S^\tau\subset\mathcal W$ for every $\tau\in(0,\tau_1)$, then
      $\sup_{a\in S^\tau}\dist_X(a,S)\to0$ as $\tau\to0$.
    \item Under the assumptions of 2. and 3., $S^\tau\to S$ in the
      Kuratowski sense and $d_H(S^\tau,S)\to0$.
  \end{enumerate}
\end{lemma}

\end{document}